\documentclass[11pt, reqno]{amsart}
\usepackage{cite}
\usepackage{wrapfig, lipsum,booktabs}
\usepackage{graphicx}
\usepackage{amssymb}
\usepackage{epstopdf}
\usepackage{verbatim}
\usepackage{bm,esvect}
\usepackage{multicol}
\usepackage{multirow}
\usepackage{subfigure}
\usepackage{fancyhdr} 
\usepackage{float}
\usepackage{stmaryrd}
\usepackage{color}
\usepackage{soul}
\usepackage{tikz}
\usepackage{todonotes}
\usetikzlibrary{patterns}
\usepackage{pgfplots}
\usepackage{cancel}
\newtheorem{theorem}{Theorem}[section]

\newtheorem{lemma}{Lemma}[section]
\newtheorem{remark}{Remark}[section]

\usepackage{multirow}
\usepackage{amsmath,amssymb,eucal}
\usepackage{graphicx,subfigure,epsfig}
\usepackage{psfrag}
\usepackage{url}
\usepackage[top=1in, bottom=1.in, left=1in, right=1in]{geometry}

\newcommand{\bld}[1]{\hbox{\boldmath$#1$}}    
\newcommand{\Th}{\mathcal{T}_h}

\newcommand{\Eh}{\mathcal{E}_h}

\newcommand{\What}{\widehat W_{h,0}^{k-1}}
\newcommand{\Wc}{\mathcal {W}_{h,0}^{1}}
\newcommand{\Wh}{W_{h}^{k}}

\newcommand{\Vhat}{\widehat {\bld V}_{h,0}^{k-1}}
\newcommand{\Vhato}{\widehat {\bld V}_{h}^{k-1}}
\newcommand{\Vc}{\bld{\mathcal {V}}_{h,0}^{1}}
\newcommand{\Vh}{\bld V_{h,0}^{k}}
\newcommand{\Vhd}{\bld V_{h,0}^{k,\partial}}
\newcommand{\Vhos}{\bld V_{h}^{k,o*}}
\newcommand{\Vho}{\bld V_{h}^{k,o}}
\newcommand{\Vhr}{\bld V_{h,0}^{k,\mathrm{cst}}}
\newcommand{\Sh}{X_{h,0}^{k+1}}

\newcommand{\tang}{\mathsf{tang}}
\newcommand{\vhat}{\widehat {\bld v}_{h}}

\newcommand{\vh}{\bld v_{h}}
\newcommand{\vhd}{\bld v_{h}^\partial}
\newcommand{\vho}{\bld v_{h}^o}

\newcommand{\whd}{\bld w_{h}^\partial}
\newcommand{\what}{\widehat {\bld w}_{h}}

\newcommand{\uhat}{\widehat {\bld u}_{h}}

\newcommand{\uh}{\bld u_{h}}
\newcommand{\uhd}{\bld u_{h}^\partial}
\newcommand{\vph}{\vv{\nabla}\times \phi_h}
\newcommand{\vqh}{\vv{\nabla}\times \psi_h}
\newcommand{\ph}{\phi_{h}}
\newcommand{\qh}{\psi_{h}}

\setulcolor{red}
\newcommand{\jmp}[1]{[\![#1 ]\!]}
\newcommand{\avg}[1]{\{\!\!\{#1 \}\!\!\}}

\begin{document}
\title[ASP-HDG]{Uniform auxiliary space preconditioning for HDG methods for 
elliptic operators with a parameter dependent low order term}
\author{Guosheng Fu}
\address{Department of Applied and Computational Mathematics and 
Statistics, University of Notre Dame, USA.}
\email{gfu@nd.edu}
 \thanks{We gratefully acknowledge the partial support of this work
 from U.S. National Science Foundation through grant DMS-2012031.}

\begin{abstract}
  The auxiliary space preconditioning (ASP) technique is 
applied to the HDG schemes for three different elliptic 
problems with a parameter dependent low order term, namely, 
a symmetric interior penalty HDG scheme for the scalar reaction-diffusion equation, 
a divergence-conforming HDG scheme for a vectorial reaction-diffusion equation, 
and a $C^0$-continuous interior penalty HDG scheme for the generalized biharmonic equation with a low order term.
Uniform preconditioners are obtained for each case and the general ASP theory
by J. Xu \cite{Xu96} is used to prove the optimality with respect to the mesh size and
uniformity with respect to the low order parameter. 
\end{abstract}
\maketitle

\section{Introduction}
\label{sec:intro}
The hybridizable discontinuous Galerkin (HDG) methods were originally introduced 
about a decade ago in \cite{CockburnGopalakrishnanLazarov09} for diffusion problems as a subclass 
of discontinuous Galerkin finite element methods  
amenable to {\it static condensation} and hence to efficient implementation.
Since then, the HDG methodology has been successfully applied to a variety of
problems in computation fluid dynamics and continuum mechanics \cite{Cockburn01,
Cockburn02, Nguyen03, Cockburn03}. 

Compared with the fast development of different HDG schemes for various partial
differential equations (PDEs), there are relative few work on the issue of linear
system solvers for the resulting condensed HDG system.
The first investigation on fast HDG solvers appeared in \cite{CockburnDubois14}
where the authors presented a geometric multigrid algorithm for an HDG scheme for the
diffusion problem. We also mention the very recent work \cite{Bui20} on a multilevel HDG preconditioner
 that is applicable to various PDEs, and refer to \cite[Section 3.1]{Bui20} 
for a short review on existing work on solvers/preconditioners for HDG systems.

The auxiliary space preconditioning (ASP) method was 
first proposed by Xu in \cite{Xu96}, which givens an optimal (Poisson-based)
preconditioner 
for symmetric positive definite (SPD) linear systems; see the review \cite{Xu10}.
Roughly speaking, the main idea of ASP is to use 
well-established fast Poisson solvers  as building blocks to
develop {\it user-friendly} solvers for various discretized PDEs \cite{Xu10}.
ASP has been applied to the HDG method in \cite{LiXie16} 
and the closely related weak Galerkin method in \cite{Chen15}, both for the pure
diffusion problem.

In this work, we apply ASP to HDG schemes for three different elliptic 
problems with a parameter dependent low order term.
In particular, we consider a symmetric interior penalty HDG scheme for the
scalar reaction-diffusion equation, a divergence-conforming HDG scheme 
for a vectorial reaction-diffusion equation, and a $C^0$-continuous interior
penalty HDG scheme for the generalized biharmonic equation with a low order term.
Uniform preconditioners are obtained for each case and the general ASP theory
\cite{Xu96} is used to prove the optimality with respect to the mesh size and
uniformity with respect to the low order parameter. We do not theoretically
address the issue of robustness of the preconditioner with respect to the polynomial
degree, but our numerical results indicate the growth of the iteration counts of
the preconditioned conjugate gradient (PCG) algorithm on the polynomial degree is quite mild.
We point out that while our analysis is focused on interior penalty type  HDG schemes 
 for each problem, our results are easily applicable to other HDG formulations and other
hybrid finite element methods.

For the scalar and vectorial reaction-diffusion problems, we use the linear
continuous finite element space as the auxiliary space, whose associated linear
systems is precondioned by hypre's BoomerAMG preconditioner \cite{hypre,
Henson02}. 
For the generalized biharmonic problem, we borrow the idea from the 
preconditioner for a divergence-free DG scheme for the Stokes 
problem in \cite{Ayuso14}, where we take a larger auxiliary space 
whose associated 
linear system is easier to precondition than the original system for the
generalized  biharmonic
problem. Due to a technical difficulty, the robustness of the preconditioner 
for the generalized biharmonic problem can
only be proven for the simply supported boundary condition case, 
which corresponds to the slip boundary condition for Stokes flow discussed in 
\cite{Ayuso14}.  Our numerical
results confirm that our proposed preconditioner does deteriorate in
performance with a mesh dependent convergence behavior
when applied to the clamped boundary condition case.
The construction of an optimal preconditioner 
for the generalized biharmonic problem with a clamped boundary condition is 
the subject of a forthcoming paper.

The rest of the paper is organized as follows.
In Section \ref{sec:asp}, we briefly review the ASP theory in \cite{Xu96}.
In Section \ref{sec:hdg}, we construct ASP for the aforementioned
HDG schemes and apply the ASP theory to prove their optimality.
In Section \ref{sec:num}, we present numerical examples to 
support the results in Section \ref{sec:hdg}. 
Finally, we conclude in Section \ref{sec:conclude}.

\section{The auxiliary space preconditioning theory}
\label{sec:asp}
In this section, we review the abstract ASP theory developed in 
\cite{Xu96}. We closely follow the discussion in \cite[Section 2]{Xu96}.

The ASP technique can be interpreted as a two level non-nested 
multigrid preconditioner. 
To be specific, we consider a finite dimensional 
linear inner product space $\mathcal{V}$ with an inner product $(\cdot , \cdot)$, 
and a linear symmetric positive definite (SPD) operator $A:\mathcal{V}\rightarrow
\mathcal{V}$ with respect to this inner product. The ASP technique serves to
construct an optimal (SPD) preconditioner  $B$ for the operator $A$ so that
mesh independent convergence behavior can be obtained for 
the 
PCG algorithm with
the preconditioner $B$ for solving the linear
system problem $$Au=f\in \mathcal{V}.$$
The main ingredient of the  ASP technique is an auxiliary linear inner product space 
$\mathcal{V}_0$ together with an operator $A_0:\mathcal{V}_0\rightarrow
\mathcal{V}_0$ that is SPD with respect to an inner product $[\cdot, \cdot]$
on $\mathcal{V}_0$. The auxiliary space $\mathcal{V}_0$ is chosen in such a way
that the operator $A_0$ can be more easily preconditioned than $A$.

The auxiliary space preconditioner in \cite{Xu96} takes the following 
additive form:
\begin{align}
  \label{asp-prec}
  B = R+\Pi\, B_0\,\Pi^t,
\end{align}
where  $B_0:\mathcal{V}_0\rightarrow\mathcal{V}_0$ 
is an {\it optimal} preconditioner for $A_0$, 
$\Pi:\mathcal{V}_0\rightarrow\mathcal{V}$
is an operator that links the two spaces and 
$\Pi^t$ is the {\it adjoint} operator defined by 
\begin{align*}
  [\Pi^t v, w] = (v, \Pi w)\quad v\in \mathcal{V}, w\in \mathcal{V}_0,
\end{align*}
and $R:\mathcal{V}\rightarrow\mathcal{V}$ is an SPD operator serves to
resolve what can not be resolved by the auxiliary space.
If $\mathcal{V}_0$ is viewed as a {\it coarse} space in multigrid terminology, 
$\Pi$ can be considered as the prolongation operator, 
$\Pi^t$ is the restriction operator, $R$ is the smoother, and $B_0$ is the
coarse grid solver.
In most applications, the smoother $R$ is given by a simple relaxation scheme such as 
Jacobi or symmetric Gauss-Seidel method.

We cite verbatimly below  the main
abstract result on ASP theory
\cite[Theorem 2.1]{Xu96}, which gives a sufficient condition for the optimality
of the preconditioner \eqref{asp-prec}.
\begin{theorem}[Theorem 2.1 of \cite{Xu96}]
  \label{thm:asp}
  Assume that there are some non-negative constants $\alpha_0, \alpha_1, 
  \lambda_0, \lambda_1$ and $\beta_1$ such that, for all $v\in \mathcal{V}$ and 
  $w\in \mathcal{V}_0$, 
  \begin{subequations}
    \label{asp-assum}
    \begin{align}
      \label{asp-as1}
    \alpha_0\rho_A^{-1}(v,v)\le &(Rv,v)\le \alpha_1 \rho_A^{-1} (v,v), \\
      \label{asp-as2}
    \lambda_0 [w,w]_{A_0}\le &[B_0A_0w, w]_{A_0}\le \lambda_1 [w,w]_{A_0},\\
      \label{asp-as3}
    \|\Pi w\|_A^2\le & \beta_1 \|w\|_{A_0}^2, 
    \end{align}
  and furthermore, assume that there exists a linear operator
  $P:\mathcal{V}\rightarrow \mathcal{V}_0$ and positive constants 
  $\beta_0$ and $\gamma_0$ such that, 
  \begin{align}
      \label{asp-as4}
    \|Pv\|_{A_0}^2\le \beta_0^{-1} \|v\|_A^2
  \end{align}
  and 
  \begin{align}
      \label{asp-as5}
    \|v-\Pi Pv\|^2\le \gamma_0^{-1}\rho_A^{-1} \|v\|_A^2.
  \end{align}
  Then the preconditioner given by \eqref{asp-prec}
satisfies 
\begin{align*}
  \kappa(BA)\le (\alpha_1+\beta_1\lambda_1)((\alpha_0\gamma_0)^{-1}
  +(\beta_0\lambda_0)^{-1}).
\end{align*}
In particular, if $P$ is a right inverse of $\Pi$ namely $\Pi P v = v$
for $v\in \mathcal V$, then 
\begin{align*}
\kappa((\Pi\,B_0\,\Pi^t)A)\le \frac{\beta_1}{\beta_0}\frac{\lambda_1}{\lambda_0}.
\end{align*}
  \end{subequations}
\end{theorem}
Here $\rho_A= \rho(A)$ denotes the spectral radius of $A$, $\|\cdot\|$ denotes
the norm induced by $(\cdot, \cdot)$,
the notation
$(\cdot, \cdot)_A:=(A\cdot, \cdot)$ 
and $[\cdot, \cdot]_{A_0} := [A_0\cdot, \cdot]$, 
and $\kappa(A)$ is the condition number of the operator $A$.


The focus of the next section is on the 
application of Theorems \ref{thm:asp} to construct 
optimal auxiliary space preconditioners for the three HDG schemes
under consideration. Before proceeding to the details, 
we make a remark below on steps in the ASP
construction for the scalar and vectorial reaction-diffusion case.
The ASP construction for the biharmonic case follows a different route, whose
discussion is postponed to Subsection \ref{sec:bh}.
\begin{remark}[ASP for HDG schemes for reaction-diffusion systems]
  \label{rk1}
  An HDG scheme consists of two sets of degrees of freedom (DOFs):
  the {\it local} DOFs inside the mesh elements 
  that can be statically condensed out, and the {\it global} DOFs on the mesh
  skeleton which shall form a global linear system to be solved.
  We denote the finite element space related to the global DOFs as $V_h$, 
  and the associated HDG bilinear form on $V_h$ as $a_h(\cdot, \cdot)$.
  The construction of ASP
  for the HDG schemes for scalar and vectorial reaction-diffusion equations 
  is performed as follows:
  \begin{itemize}
    \item [\textbf{(i)}] 
      Define an $L^2$-like inner-product $(\cdot, \cdot)_{0,h}$ 
for the global HDG space $V_h$, and estimate the
spectral radius  $\rho_{A_h}$ of the condensed HDG operator 
  $A_h:V_h\rightarrow V_h$ where 
  $$
  (A_hu, v)_{0,h} = a_h(u, v)\quad \forall u,v\in V_h.
  $$
\item [\textbf{(ii)}] Prove the simple smoother $R_h$ given 
  by the point Jacobi method satisfies \eqref{asp-as1}. 
  In practical implementation, we also use the slightly more sophisticated 
  block Gauss-Seidel method 
  for the smoother to reduce PCG iterations for
  convergence.
\item [\textbf{(iii)}] Use the continuous piecewise linear finite element space $V_{h,0}$
  as the auxiliary space, and denote the associated bilinear form 
  for the  reaction-diffusion operator 
  as 
  $a_{h,0}(\cdot,\cdot)$,
  with the corresponding linear operator $A_{h,0}$
  given as follows 
  $$
  [A_{h,0}u_0, v_0] = a_{h,0}(u_0, v_0)\quad \forall u_0,v_0\in V_{h,0},
  $$
  where $[\cdot, \cdot]$ is the usual $L^2$-norm on $V_{h,0}$.
  Construct a robust preconditioner $B_{h,0}$
  for the operator $A_{h,0}$ such that the estimate
  \eqref{asp-as2} holds.  
\item [\textbf{(iv)}] Construct the operators $\Pi_h: V_{h,0}\rightarrow V_{h}$
  and 
$P_h: V_{h}\rightarrow V_{h,0}$
  such that 
  the stability properties  \eqref{asp-as3}--\eqref{asp-as4} and the approximation property 
\eqref{asp-as5} holds.
\item [\textbf{(v)}] Apply Theorem \ref{thm:asp} to conclude the 
  optimality of the
  preconditoiner 
$  B_h = R_h+\Pi_h\, B_{h,0}\,\Pi_h^t
$
  for the operator $A_h$.
  \end{itemize}
\end{remark}

\section{ASP for HDG schemes}
\label{sec:hdg}
\subsection{Preliminaries and finite element spaces}
\label{sec:rd}
We consider a convex polygonal/polyhedral domain $\Omega \subset \mathbb{R}^d, d=2,3$.
The restrictions on the domain are only used for theoretical analysis. In
practical implementation, we can work with more general non-convex or curved domains.

Let $\Th$ be a shape-regular, quasi-uniform conforming simplicial triangulation of the domain
$\Omega$. 
For any element $K\in \Th$, we denote by $h_K$ its diameter, and
by $h$ the maximum diameter over all mesh elements.
We denote $\Eh$ as the set of facets (edges in 2D, faces in 3D)
on the mesh $\Th$, which we also refer to as the {\it mesh skeleton}.
We split $\Eh$ into the collection of boundary facets 
$\Eh^\partial=\{F\in \Eh: \; F\subset \partial\Omega \}$, 
and that of the interior facets $\Eh^o=\Eh\backslash \Eh^\partial$.
Given a simplex $S\subset \mathbb{R}^d, d=1,2,3$, we denote 
$\mathcal{P}^m(S)$, $m\ge 0$, as the space of polynomials of degree at most
$m$. 
Given a facet $F\in \Eh$ with normal direction $\bld n$, we denote 
$\mathsf{tang}(\bld w):=\bld w-(\bld w\cdot\bld n)\bld n$ as the {\it
tangential component} of a vector field $\bld w$. 

The following finite element spaces will be used to construct the HDG scheme
and ASP for the scalar reaction-diffusion equation:
\begin{subequations}
  \label{spaces-rd}
\begin{align}
  \label{l2v}
    W_h^{r}:=&\;\{w\in L^2(\Omega): \;\; w|_K\in
  \mathcal{P}^r(K), \;\forall K\in \Th\}, \\ 
  \label{l2f}
      \widehat{W}_h^{r}:=&\;\{\widehat{w}\in L^2(\Eh): \;\; \widehat{w}|_F\in
  \mathcal{P}^r(F), \;\forall F\in \Eh\}, \\ 
  \label{h1}
        \mathcal{W}_h^1:=&\;\{w\in H^1(\Omega): \;\; w|_K\in
  \mathcal{P}^1(K), \;\forall K\in \Th\}. 
\end{align}
where $r\ge 0$ is the polynomial degree.
We further denote the subspaces with vanishing boundary conditions: 
\begin{align}
  \label{l2f0}
  \widehat{W}_{h,0}^{r}:=&\;\{\widehat{w}\in 
   \widehat{W}_{h}^{r}: \;\; \widehat{w}|_F=0, \;\forall F\in \Eh^\partial\},
   \;\;\;\;
  \mathcal{W}_{h,0}^{1}:=\;\{{w}\in 
\mathcal{W}_{h}^{1}: \;\; {w}|_{\partial\Omega}=0\}.
\end{align}
\end{subequations}

The following finite element spaces will be used to construct the 
divergence-conforming HDG scheme
and ASP for the vectorial reaction-diffusion equation:
\begin{subequations}
  \label{spaces-vrd}
\begin{align}
  \label{vl2v}
  \bld V_h^{r}:=&\;\{\bld v\in H(\mathrm{div};\Omega): \;\; \bld v|_K\in
  [\mathcal{P}^r(K)]^d, \;\forall K\in \Th\}, \\ 
  \label{vl2f}
        \widehat{\bld V}_h^r:=&\;\{\widehat{\bld v}
          \in [L^2(\Eh)]^d: \;\;\widehat{\bld v}|_F\in
        [\mathcal{P}^r(F)]^d,\;\;\widehat{\bld v}\cdot \bld n|_F=0, 
      \;\forall F\in \Eh\}, \\ 
  \label{vh1}
          \bld{\mathcal{V}}_h^1:=&\;\{\bld v\in [H^1(\Omega)]^d: \;\; \bld v|_K\in
          [\mathcal{P}^1(K)]^d, \;\forall K\in \Th\}. 
\end{align}
We further denote the following subspaces with vanishing boundary conditions:
\begin{align}
  \label{vl2v0}
  \bld{V}_{h,0}^{r}:=&\{\bld{v}\in 
  \bld{V}_{h}^{r}: \; \bld{v}\cdot\bld n|_{\partial\Omega}=0
\},\\
  \label{vl2f0}
    \widehat{\bld{V}}_{h,0}^{r}:=&\{\widehat{\bld{v}}\in 
      \widehat{\bld{V}}_{h}^{r}: \; \mathsf{tang}(\widehat{\bld{v}})|_{F}=0,
      \forall F\in\Eh^\partial
\},\\
  \label{vh10}
      \bld{\mathcal{V}}_{h,0}^{1}:=&\;\{{\bld v}\in 
      \bld{\mathcal{V}}_{h}^{1}: \;\; {\bld v}|_{\partial\Omega}=0\}.
\end{align}
Moreover, we shall also use a constraint subspace of $\bld V_h^r$ 
whose elements has a piecewise constant divergence field:
\begin{align}
  \label{vl2vc}
  \bld{V}_{h}^{r, \mathsf{cst}}:=&\{\bld{v}\in 
    \bld{V}_{h}^{r}: \; \nabla\cdot \bld{v}|_{K}\in \mathcal{P}^0(K),
    \forall K\in \Th\},\\
  \label{vl2vc0}
  \bld{V}_{h,0}^{r, \mathsf{cst}}:=&\{\bld{v}\in 
    \bld{V}_{h,0}^{r,\mathsf{cst}}: \; \bld{v}\cdot \bld n|_{\partial\Omega}=0
  \}.
  \end{align}
  
  We perform a {\it hierarchical} basis splitting for the divergence-conforming 
  space $\bld V_h^r$ (and $\bld V_h^{r,\mathsf{cst}}$) 
  as was done in \cite[Chapter 5]{Zaglmayr}, see
  also \cite[Section 2.2.4]{Lehrenfeld10},  to facilitate the discussion on static condensation.
  The following basis splitting was presented in \cite[Table
  2.1]{Lehrenfeld10}\footnote{The basis splitting in \cite{Lehrenfeld10}
  is only performed on triangular meshes, 
similar results on other element shapes were documented in \cite{Zaglmayr}.
These hierarchical bases are readily available in the 
NGSolve software\cite{Schoberl16, ngsolve}, which we use in
our numerical simulations.
}:
\begin{align}
  \label{hdiv-split}
  \bld V_h^r = 
  \bigoplus_{\tiny F\in\Eh} \mathrm{span}\{\bld \phi_F^0\}
\oplus 
\!\!\!\!\! 
\!\!\!\!\! 
  \bigoplus_{
    \tiny
    \begin{tabular}{c}
    $F\in\Eh$\\
    $i=1,\cdots, n_f^r$
  \end{tabular}
}
\!\!\!\!\! 
\!\!\!\!\! 
\mathrm{span}\{\bld \phi_F^i\}
\oplus
\!\!\!\!\! 
\!\!\!\!\! 
\bigoplus_{
    \tiny
    \begin{tabular}{c}
    $K\in\Th$\\
    $i=1,\cdots, n_k^{r,1}$
  \end{tabular}
}
\!\!\!\!\! 
\!\!\!\!\! 
\mathrm{span}\{\bld \phi_K^{i}\}
\oplus 
\!\!\!\!\! 
\!\!\!\!\! 
\bigoplus_{
    \tiny
    \begin{tabular}{c}
    $K\in\Th$\\
    $i=1,\cdots, n_k^{r,2}$
  \end{tabular}
} 
\!\!\!\!\! 
\!\!\!\!\! 
\mathrm{span}\{\bld \psi_K^{i}\},
\end{align}
where $\bld \phi_F^0$ is the (global) lowest-order basis for
the lowest-order Raviart-Thomas space whose normal component is only supported
on the facet $F$,
$\bld \phi_F^i$,
$i=1,\cdots, n_f^r$, 
is the 
(global) higher order divergence-free facet bubble basis whose normal component 
is only supported on the facet $F$, 
$\bld \phi_K^i$, 
$i=1,\cdots, n_k^{r,1}$,
is the (local)
higher order divergence-free element bubble 
basis 
that is supported only on the element $K$
whose normal component vanishes on the mesh skeleton $\Eh$ and 
$\bld \psi_K^i$, $i=1,\cdots, n_k^{r,2}$, is the (local)
higher order element bubble basis
with a non-zero divergence
that is supported only on the element $K$
whose normal component vanishes on the mesh skeleton $\Eh$.
Here the integer numbers $n_f^r$, $n_k^{r,1}$,
and $n_k^{r,2}$ are the number of basis functions per facet/element for each group
of basis functions, whose specific values 
can be found in \cite{Zaglmayr} which are not relevant in our discussion.
With this basis splitting, we now define the {\it global} and {\it local}
subspaces of $\bld V_h^r$: 
\begin{align}
  \label{hdiv-split-0}
  \bld V_h^{r,\partial} := &\;
  \bigoplus_{\tiny F\in\Eh} \mathrm{span}\{\bld \phi_F^0\}
\oplus 
\!\!\!\!\! 
\!\!\!\!\! 
  \bigoplus_{
    \tiny
    \begin{tabular}{c}
    $F\in\Eh$\\
    $i=1,\cdots, n_f^r$
  \end{tabular}
}
\!\!\!\!\! 
\!\!\!\!\! 
\mathrm{span}\{\bld \phi_F^i\}\\
  \label{hdiv-split-1}
  \bld V_h^{r,o} :=&\; 
\!\!\!\!\! 
\!\!\!\!\! 
\bigoplus_{
    \tiny
    \begin{tabular}{c}
    $K\in\Th$\\
    $i=1,\cdots, n_k^{r,1}$
  \end{tabular}
}
\!\!\!\!\! 
\!\!\!\!\! 
\mathrm{span}\{\bld \phi_K^{i}\},
\;\;\;\;\;
\;\;\;\;\;
  \bld V_h^{r,o*} :=\; 
  \bld V_h^{r,o}  
\oplus 
\!\!\!\!\! 
\!\!\!\!\! 
\bigoplus_{
    \tiny
    \begin{tabular}{c}
    $K\in\Th$\\
    $i=1,\cdots, n_k^{r,2}$
  \end{tabular}
} 
\!\!\!\!\! 
\!\!\!\!\! 
\mathrm{span}\{\bld \psi_K^{i}\}.
\end{align}
We further denote $\bld V_{h,0}^{r,\partial}$ as the 
subspace of $\bld V_{h}^{r,\partial}$ whose normal component vanishes on 
the domain boundary $\partial \Omega$.
Finally, we have the following direct decomposition of the divergence-conforming 
spaces that will be used in static condensation:
\begin{align}
  \label{vfe}
  \bld V_h^r =  \bld V_h^{r,\partial} \oplus
  \bld V_h^{r,o*},\;\;\;
  \bld V_h^{r,\mathsf{cst}} =  \bld V_h^{r,\partial} \oplus
  \bld V_h^{r,o}.
\end{align}
\end{subequations}

The following additional high-order $H^1$-conforming 
finite element space (in two dimensions) 
will be used to construct the $C^0$-continuous interior penalty HDG scheme
for the generalized biharmonic equation: 
  \begin{align}
  \label{bh-space-1}
  X_h^r:=&\;\{\phi\in H^1(\Omega): \;\; \phi|_K\in
  \mathcal{P}^r(K), \;\forall K\in \Th\}.
  \end{align}
  We further denote $X_{h,0}^r$ as the subspace of $X_h^r$ with vanishing
  boundary conditions on the domain boundary $\partial\Omega$.
  Denote $V_{h,0}^{r,\mathrm{zero}}$ as the divergence-free subspace of 
  $V_{h,0}^{r}$,
  it is well-known that in two dimensions, the following equality holds:
  \[
    V_{h,0}^{r, \mathrm{zero}}  =  \{\vv{\nabla} \times \phi:\;
    \phi\in X_{h,0}^{r+1}\},
  \]
  where $\vv{\nabla}\times \phi =(\partial_y \phi, -\partial_x \phi)$ is the curl operator
  in two dimensions.

To simplify notation, 
for any function $\phi\in L^2(S)$
we denote $\|\phi\|_S$ as the $L^2$-norm of $\phi$ on the domain 
$S$, and for two positive numbers $A, B\in \mathbb{R}^+$, we denote $A\lesssim B$ to indicate that there exists
a generic positive constant $C$ that is only dependent on the 
shape regularity of the mesh $\Th$ and the polynomial degree of the
finite element space under consideration such that 
$A\le C B$. In particular, the hidden constant $C$ is independent of the mesh size $h$
and parameters in the PDEs. Furthermore, we denote $A \simeq B$ to indicate
that $A\lesssim B$ and $B\lesssim A$.

\subsection{Symmetric interior penalty HDG for reaction-diffusion}
\label{sec:rd}
\subsubsection{The model and the HDG scheme}
We consider the 
following constant-coefficient reaction-diffusion equation with a homogeneous Dirichlet 
boundary condition:
\begin{align}
  \label{rd-eq}
  -\triangle u + \tau u = f \;\; \text{in } \Omega, \quad\quad 
u|_{\partial \Omega}=0
\end{align}
where $\tau\ge 0$ is the non-negative reaction coefficient.
Here we focus on the analysis for this constant coefficient case. 
Variable coefficients will be covered in the numerical experiments only. 

We apply the symmetric interior penalty HDG scheme with projected jumps 
\cite[Remark 1.2.4]{Lehrenfeld10}
to discretize the equation \eqref{rd-eq}. Given a polynomial degree $k\ge 1$, 
the HDG scheme reads as follows:
find $(u_h, \widehat u_h)\in W_h^k\times \widehat{W}_{h,0}^{k-1}$ such that 
  \begin{align}
  \label{hdg-rd}
  a_h\left((u_h, \widehat{u}_h), (v_h, \widehat{v}_h)
    \right) = \int_{\Omega} f\,v_h\mathrm{dx},\quad
    \forall (v_h, \widehat{v}_h)\in W_h^{k}\times \widehat{W}_{h,0}^{k-1},
  \end{align}
where the bilinear form 
\begin{align*}
  a_h((u_h, \widehat{u}_h), (v_h, \widehat{v}_h))
  := \sum_{K\in\Th}&\; \int_{K}(\nabla u_h\cdot\nabla v_h+\tau u_hv_h)\mathrm{dx}
  -\int_{\partial K} \nabla u_h\cdot \bld n(v_h-\widehat v_h)\mathrm{ds}\\
  &\!\!\!\!\!\!\!\!\!\!\!\!\!\!\!\!\!\!
  -\int_{\partial K} \nabla v_h\cdot \bld n(u_h-\widehat u_h)\mathrm{ds}
  +\int_{\partial K}\frac{\alpha k^2}{h} (P_{k-1} u_h-\widehat u_h)(P_{k-1}
  v_h-\widehat v_h)\mathrm{ds},
\end{align*}
with $P_{k-1}$ denoting the $L^2$-projection into the space of facet-piecewise
polynomials of degree $k-1$. Here the stability parameter $\alpha$ is 
chosen big enough to ensure the following 
coercivity result:
\begin{align}
\label{hdg-rd-co}
  \|(u_h, \widehat u_h)\|_{1,h}^2  \lesssim  
  a_h((u_h, \widehat{u}_h), (u_h, \widehat{u}_h)),
\end{align} 
where 
\[
  \|(u_h, \widehat u_h)\|_{1,h}^2:=\sum_{K\in\Th}(\|\nabla u_h\|_K^2
  +\tau \|u_h\|_K^2
  +\frac{1}{h}\|u_h-\widehat{u}_h\|_{\partial K}^2)
\]
is a norm on $W_h^k\times \widehat{W}_{h,0}^{k-1}$.
A lower bound on $\alpha$ that guarantees coercivity was presented 
in \cite[Lemma 1]{AinsworthFu18}. In practice, taking
$\alpha = 4$ ensures \eqref{hdg-rd-co}, which is the value we use  
in the numerical simulations.
On the other hand, the Cauchy-Schwarz inequality and inverse inequality
ensure that 
$a_h((u_h, \widehat{u}_h), (u_h, \widehat{u}_h))\lesssim \|(u_h, 
\widehat{u}_h\|_{1,h}^2 $ for $(u_h, \widehat{u}_h)\in W_h^k\times 
W_{h,0}^{k-1}$. Hence, the bilinear form $a_h$ induces a norm 
that is equivalent to $\|\cdot \|_{1,h}$:
\begin{align}
  \label{hdg-rd-norm}
  a_h((u_h, \widehat{u}_h), (u_h, \widehat{u}_h))
  \simeq \|(u_h, \widehat u_h)\|_{1,h}^2, \quad\forall 
(u_h, \widehat{u}_h)\in \Wh\times \What.
\end{align} 
\subsubsection{Static condensation}
When implementing the scheme \eqref{hdg-rd}, 
static condensation is performed locally in the element-level 
to eliminate DOFs associated with 
the {\it local} space $W_h^k$, which results in a global {\it condense} linear system for 
DOFs associated with the {\it global} space $\widehat{W}_{h,0}^{k-1}$ only.
The main subject of this subsection is to construct an ASP method 
for this condensed system. To proceed, we first give a
characterization of the condensed  bilinear form.

Denoting the following lifting operator 
$\mathcal{L}_h:\widehat{W}_{h,0}^{k-1}\rightarrow W_h^k$:
given $\mu\in \widehat{W}_{h,0}^{k-1}$, 
$\mathcal{L}_h(\mu)$ is the unique function in $W_h^k$ such that
\begin{align}
  \label{rd-L1}
  a_h((\mathcal{L}_h(\mu), 0), (v_h,0)) = 
  -a_h((0, \mu), (v_h,0)) 
  ,\;\;\;\; \forall v_h\in W_h^k.
\end{align}
Furthermore, let $\mathcal{L}_h^f$ be the unique function in $W_h^k$
such that 
\begin{align}
  \label{rd-L2}
  a_h((\mathcal{L}_h^f, 0), (v_h,0)) = 
  \int_\Omega f v_h\mathrm{dx},\;\; \forall v_h\in W_h^k.
\end{align}
Well-posedness of the above linear systems easily follows from the coercivity result 
\eqref{hdg-rd-co}.

\begin{lemma}[Characterization of the condensed system for \eqref{hdg-rd}]
  \label{lemma:rd}
  Let $(u_h,\widehat u_h)\in W_h^k\times \widehat W_{h,0}^{k-1}$
be the unique solution 
  to the HDG scheme \eqref{hdg-rd}.
  Then,
  \begin{subequations}
  \begin{align}
    \label{rd-uh}
    u_h = &\;\mathcal{L}_h(\widehat u_h)+\mathcal{L}_h^f, 
  \end{align}
  and $\widehat u_h$ is the unique function in  $\What$ 
  such that 
  \begin{align}
    \label{rd-uhat}
   a_h\left((\mathcal{L}_h(\widehat u_h), \widehat u_h), 
    (\mathcal{L}_h(\widehat v_h), \widehat v_h)\right)
    = 
    -a_h((\mathcal{L}_h^f, 0), 
    (\mathcal{L}_h(\widehat v_h), \widehat v_h)),\quad \forall \widehat v_h\in \widehat W_{h,0}^{k-1}.
  \end{align}
  In particular, the reduced system \eqref{rd-uhat} is symmetric and positive
  definite, and 
  \begin{align}
    \label{norm-eq}
   a_h\left((\mathcal{L}_h(\mu), \mu), 
    (\mathcal{L}_h( \mu), \mu)\right)
    \simeq 
\|(\mathcal{L}_h(\mu), \mu)\|_{1,h}^2,\quad \forall \mu\in \What.
  \end{align}
  \end{subequations}
\end{lemma}
\begin{proof}
  Taking test functions $(v_h, 0)$ in the HDG scheme \eqref{hdg-rd}, 
  we immediately obtain
  the equality \eqref{rd-uh} from the definitions
  \eqref{rd-L1}--\eqref{rd-L2} and linearity of the bilinear form 
  $a_h$.
  Taking test functions $(\mathcal L_h(\widehat v_h), \widehat v_h)$
  in \eqref{hdg-rd} and reordering terms, we obtain the equation 
\eqref{rd-uhat}.
Finally, the equivalence \eqref{norm-eq} is a consequence of that for the
operator  $a_h$ in \eqref{hdg-rd-norm}.
\end{proof}

\subsubsection{The auxiliary space preconditioner}
Now, we follow Remark \ref{rk1} to construct the auxiliary space perconditioner 
for the reduced HDG system \eqref{rd-uhat}.

\textbf{(i)} We denote the following $L^2$-like inner product on $\What$: 
\begin{align}
  \label{rd-l2}
  (\lambda, \mu)_{0,h}:= 
  \sum_{K\in\Th}h\int_{\partial K}\lambda\, \mu\mathrm{ds}=
  \sum_{F\in\Eh^o}h
  \int_{F}\lambda\,\mu\mathrm{ds}
  ,\quad \forall \lambda,\mu\in\What,
\end{align}
whose induced norm is denoted as $\|\mu\|_{0,h}^2:=(\mu, \mu)_{0,h}$.
We define the condensed HDG operator $A_h:\What\rightarrow\What$ as 
\begin{align}
  \label{rd-ah}
  (A_h \widehat u_h, \widehat v_h)_{0,h}:=
   a_h\left((\mathcal{L}_h(\widehat u_h), \widehat u_h), 
    (\mathcal{L}_h( \widehat v_h), \widehat v_h)\right), 
    \quad \forall \widehat u_h, \widehat v_h\in\What.
\end{align}
The following result gives an estimation of the 
spectral radius of $A_h$ and its condition number.
\begin{lemma}[Spectral radius and condition number]
  \label{lemma:rd1}
  Let $A_h$ be the operator given in \eqref{rd-ah}. Then,
  there holds $\rho_{A_h}\simeq h^{-2}$ and 
  $\kappa (A_h) \lesssim \frac{h^{-2}}{\min\{\tau+1,\, h^{-2}\}}$.
\end{lemma}
\begin{proof}
  First, we prove $\rho_{A_h}\lesssim h^{-2}$.
  By the definition of the operator $A_h$, we have 
  $(A_h \widehat u_h, \widehat u_h)_{0,h}\simeq 
  \|(\mathcal{L}_h(\widehat u_h), 
  \widehat u_h)\|_{1,h}^2.$
  Taking $\mu=\widehat u_h$ and $v_h = \mathcal {L}_h(\widehat u_h)$
  in \eqref{rd-L1} and applying the Cauchy-Schwarz inequality, we have 
  \begin{align*}
  a_h((\mathcal{L}_h(\widehat u_h), 0), (\mathcal{L}_h(\widehat u_h)
  ,0)) = &\; 
  -a_h((0, \widehat u_h), (\mathcal{L}_h(\widehat u_h)
  ,0))\\
  \le \;
  a_h((\mathcal{L}_h(\widehat u_h), 0), &(\mathcal{L}_h(\widehat u_h)
  ,0))^{1/2} 
  a_h((0, \widehat u_h), (0,\widehat u_h))^{1/2}.
  \end{align*}
  Hence, \[
  \|(\mathcal{L}_h(\widehat u_h), 
  0)\|_{1,h}^2
    \simeq
  a_h((\mathcal{L}_h(\widehat u_h), 0), (\mathcal{L}_h(\widehat u_h)
  ,0)) \lesssim 
  a_h((0, \widehat u_h), (0,\widehat u_h))\simeq
  \|(0, 
  \widehat u_h)\|_{1,h}^2.
  \] 
   Invoking the triangle inequality, we get 
  \[
  (A_h \widehat u_h, \widehat u_h)_{0,h}\simeq
  \|(\mathcal{L}_h(\widehat u_h), 
  \widehat u_h)\|_{1,h}^2\lesssim
  \|(0, 
  \widehat u_h)\|_{1,h}^2= h^{-2} \|\widehat u_h\|_{0,h}^2,
  \] 
  where the last equality following from the definition of these norms.
  Since the spectral radius of the SPD operator $A_h$ is 
  its maximum eigenvalue,
  $\lambda_{A_h}^{\max}$, we 
  conclude that
  $\rho_{A_h}=\lambda_{A_h}^{\max}\lesssim h^{-2}$.

  Next, we prove the condition number estimate.
  By inverse and triangle inequalities, we get, for any $\widehat u_h\in
  \What$,
  \begin{align*}
  \min\{\tau+1, h^{-2}\}\|\widehat u_h\|_{0,h}^2
  \lesssim &\;
  (\tau+1) \|\mathcal{L}_h(\widehat u_h)\|_{0,h}^2
+h^{-2} \|\mathcal{L}_h(\widehat u_h)-\widehat u_h\|_{0,h}^2\\
  \lesssim &\;
  (\tau+1) \|\mathcal{L}_h(\widehat u_h)\|_{\Omega}^2
+h^{-2} \|\mathcal{L}_h(\widehat u_h)-\widehat u_h\|_{0,h}^2\\
  \lesssim &\;\|(\mathcal{L}_h(\widehat u_h), \widehat u_h)\|_{1,h}^2
  \simeq 
  (A_h \widehat u_h, \widehat u_h)_{0,h},
  \end{align*}
  where in the last step, we invoked the 
  following discrete Poincar\'e inequality
  \[
   \|\mathcal{L}_h(\widehat u_h)\|_\Omega\lesssim
   \|(\mathcal{L}_h(\widehat u_h), \widehat u_h)\|_{1,h},
  \] 
  which holds on convex domains, see \cite[Lemma 2.1]{Arnold82}.
  Hence, the minimum eigenvalue $\lambda_{A_h}^{\min}$ of $A_h$
  satisfies $\min\{\tau+1, h^{-2}\}\lesssim 
  \lambda_{A_h}^{\min}$. This implies that 
  \[
    \kappa (A_h) = \frac{\lambda_{A_h}^{\max}}{\lambda_{A_h}^{\min}}
    \lesssim
\frac{h^{-2}}{\min\{\tau+1,\, h^{-2}\}}.
  \]

  Finally, we prove $h^{-2}\lesssim \rho_{A_h}$.
Take $\mu$ as a basis function 
of $\What$ that is supported on a single facet $F$, then
  norm equivalence of finite dimensional spaces and the
standard  scaling argument implies that 
  \begin{align}
    \label{rd-scale}
    (A_h \mu, \mu)_{0,h}\simeq h^{-2}\|\mu\|_{0,h}^2.
  \end{align}
  This concludes the proof.
\end{proof}
\begin{remark}[The case with a large reaction parameter]
  \label{rk:3.1}
  When the reaction parameter $\tau\gtrsim h^{-2}$, 
Lemma \ref{lemma:rd1} implies that $\kappa(A_h)\simeq 1$.
In this case, a simple Jacobi preconditioner is already 
an optimal one for $A_h$.
\end{remark}

\textbf{(ii)} 
With the above spectral radius estimate, we now show next that 
the simple Jacobi smoother satisfies \eqref{asp-as1}.

\begin{lemma}[Spectum of the Jacobi preconditioner]
  \label{lemma:rd2}
  Let $\{\phi_F^j: F\in \Eh^o, 1\le j\le N_f\}$ be a set of 
  orthogonal bases for $\What$ 
  with respect to the inner product $(\cdot,\cdot)_{0,h}$.
  Here $\phi_F^j$ is supported only on the facet $F$, and
  $N_f = \binom{k+d-1}{d-1}$ is the number of bases per facet.
  Denote $R_h:\What\rightarrow\What$ as the Jacobi preconditioner
  for $A_h$ associated with this basis set, i.e., the operator 
  $R_h^{-1}$ is  the diagonal component of 
  $A_h$:
  \begin{align}
    \label{jac}
    (R_h^{-1} \widehat u_h, \widehat v_h)_{0,h}=
  \sum_{F\in\Eh^o}\sum_{j=1}^{N_f}\mathsf{u}_F^j\mathsf{v}_F^j
    (A_h\phi_F^j, \phi_F^j)_{0,h},
  \end{align}
  where 
  $\widehat u_h = \sum_{F\in\Eh^o}\sum_{j=1}^{N_f}\mathsf{u}_F^j \phi_F^j$
  and 
  $\widehat v_h = \sum_{F\in\Eh^o}\sum_{j=1}^{N_f}\mathsf{v}_F^j \phi_F^j$
  with $\mathsf{u}_F^j\in\mathbb{R}$ and $\mathsf{v}_F^j\in\mathbb{R}$ being the
  basis coefficients of $\widehat u_h$ and $\widehat v_h$, respectively.
  Then, the following equivalence holds:
  \[
    (R_h \widehat u_h, \widehat u_h)_{0,h}\simeq \rho_{A_h}^{-1}
    (\widehat u_h, \widehat u_h)_{0,h}.
  \] 
\end{lemma}
\begin{proof}
  It suffice to show 
$$
(R_h^{-1} \widehat u_h, \widehat u_h)_{0,h}\simeq \rho_{A_h}
    (\widehat u_h, \widehat u_h)_{0,h}.
 $$
 By the definition \eqref{jac} and the assumption of the bases, we have 
 \begin{alignat*}{2}
(R_h^{-1} \widehat u_h, \widehat u_h)_{0,h}
=&\;\sum_{F\in\Eh^o}\sum_{j=1}^{N_f}(\mathsf{u}_F^j)^2
    (A_h\phi_F^j, \phi_F^j)_{0,h}
    \\
\simeq&\;
h^{-2}
\sum_{F\in\Eh^o}\sum_{j=1}^{N_f}(\mathsf{u}_F^j)^2
\|\phi_F^j\|_{0,h}^2 &&\quad\quad \text{by \eqref{rd-scale}}\\
=&\;
h^{-2}
\|\widehat u_h\|_{0,h}^2
\simeq \rho_{A_h}
\|\widehat u_h\|_{0,h}^2
 &&\quad\quad \text{by basis orthogonality.}\\
\end{alignat*}
This completes the proof.
\end{proof}
\begin{remark}[On the smoother]
The above result also holds if we use other basis functions for $\What$
like the nodal Lagrange basis, where we would have 
\[
\sum_{F\in\Eh^o}\sum_{j=1}^{N_f}(\mathsf{u}_F^j)^2
\|\phi_F^j\|_{0,h}^2
\simeq\; \|\widehat u_h\|_{0,h}^2
\] 
following from the norm equivalence of 
a finite dimensional space and a scaling argument.

This result also implies that the  symmetric Gauss-Seidel 
smoother satisfies \eqref{asp-as1} as it is spectrally equivalent to the 
Jacobi smoother \cite{Zikatanov08}.
In practice, we prefer to use a block version of these smoothers as they
usually lead to better convergence for the PCG algorithm. 
\end{remark}

\textbf{(iii)} 
The auxiliary space under consideration is the continuous piecewise linear
finite element space $\mathcal{W}_{h,0}^1$ in \eqref{l2f0}.
The associated bilinear form $a_{h,0}$ and the linear operator 
$A_{h,0}$
are given as follows:
\begin{align}
  \label{rd-ah0}
  [A_{h,0}u_0, v_0] = a_{h,0}(u_0, v_0) := \int_{\Omega}(\nabla u_0\cdot\nabla v_0+\tau
  u_0v_0)\mathrm{dx},\quad\quad
  \forall u_0,v_0\in \Wc,
\end{align}
where $[\cdot,\cdot]$ is the usual $L^2$-norm on $\Wc$.
The construction of a robust preconditioner $B_{h,0}$ 
for the operator $A_{h,0}$
is a well-studied subject in the literature.
Here we take $B_{h,0}$ to be hypre's BoomerAMG preconditioner \cite{hypre,
Henson02}, which is numerically verified to be robust with respect to 
both mesh size $h$ and reaction parameter $\tau$.

\textbf{(iv)} We take the operator $\Pi_h:\Wc\rightarrow \What$ 
to be the simple $L^2$-projection:
\begin{align}
  \label{rd-pi}
  (\Pi_h u_0, \widehat v_h)_{0,h} = (u_0, \widehat v_h)_{0,h},\quad
  \forall \widehat v_h\in \What,
\end{align}
and the operator $P_h:\What\rightarrow \Wc$ 
to be 
\begin{align}
  \label{rd-p}
P_h(\widehat u_h) =S_1(\mathcal{L}_h(\widehat u_h)),
\end{align}
where $S_1:\Wh\rightarrow \Wc$ is the simple nodal 
average:
given  any $v\in \Wh$,
$S_1(v)\in \Wc$ is defined through its vertex values as follows,
\begin{align}
  \label{s1}
  S_1(v)(\bld x_n) =\left\{ 
  \begin{tabular}{ll}
    $0$ & if $\bld x_n\in\partial \Omega$, \\[1.5ex]
    $
    \frac{1}{\#K_n}\sum_{K\subset K_n}v|_{K}(\bld x_n)
    $ & if  $\bld x_n\not\in\partial \Omega$,
  \end{tabular}
\right.
\end{align}
where $\bld x_n$ is a vertex of the mesh $\Th$, 
$K_n$ is the set of elements that have a vertex at the point $\bld x_n$,
and $\# K_n$ denotes its cardinality.

The following result shows that $\Pi_h$ and $P_h$ satisfies both the stability 
and approximation properties needed to apply the ASP theory, Theorem
\ref{thm:asp}.
\begin{lemma}[Properties of $\Pi_h$ and $P_h$]
  \label{lemma:rd3}
  Let the operator $\Pi_h$ be given in \eqref{rd-pi} and 
  $P_h$ be given in \eqref{rd-p}. Then
  the following inequalities holds
\begin{subequations}
  \label{rd-proj}
  \begin{align}
  \label{rd-proj1}
    \|\Pi_h u_0\|_{A_h}\lesssim &\;\|u_0\|_{A_{h,0}},\\
  \label{rd-proj2}
    \|P_h \widehat u_h\|_{A_{h,0}}\lesssim &\;\|\widehat u_h\|_{A_{h}},\\
  \label{rd-proj3}
    \|\widehat u_h-\Pi_h P_h \widehat u_h\|_{0,h}\lesssim &\;
    \rho_{A_h}^{-1/2}
    \|\widehat u_h\|_{A_{h}},
  \end{align}
\end{subequations}
  where 
  $$
  \|\widehat v_h\|_{A_h}^2:= 
  a_h((\mathcal{L}_h(\widehat v_h), \widehat v_h),
  (\mathcal{L}_h(\widehat v_h), \widehat v_h)),\;\;\text{and} \;\;
  \|v_0\|_{A_{h,0}}^2:= 
  a_{h,0}(v_0, v_0).
  $$
\end{lemma}
\begin{proof}
  By the definition of $\Pi_h$ in \eqref{rd-pi}, there holds
  \begin{align}
    \label{rd-p1}
    \|\Pi_h u_0-u_0\|_{0,h}
    =\inf_{v_0\in\What}\|u_0-v_0\|_{0,h}
    \le \|u_0-\bar{u}_0\|_{0,h}\lesssim 
    \|u_0-\bar{u}_0\|_{\Omega}\lesssim 
    h\|\nabla u_0\|_{\Omega},
  \end{align} 
  where $\bar{u}_0\in W_h^0$ is the $L^2$-projection of 
  $u_0$ into the piecewise constant space $W_h^0$.
 
  Next, we prove 
 \begin{align}
   \label{rd-p2}
   \|u_0\|_{A_h}\lesssim \|u_0\|_{A_{h,0}},\quad \forall u_0\in \Wc,
 \end{align}
 where  $u_0$ on the left hand side is understood as a function 
 on the mesh skeleton.
 By \eqref{norm-eq}, we have 
 \[
   \|u_0\|_{A_h}\simeq \|(\mathcal{L}_h(u_0), u_0)\|_{1,h}.
 \] 
 On the other hand, definition of the lifting operator \eqref{rd-L1} implies, 
 for any $v_h\in\Wh$, 
\begin{align*}
  a_h((\mathcal{L}_h(u_0)-u_0, 0), (v_h, 0))
  =&\; -a_h((u_0, u_0), (v_h, 0))
\end{align*}
Taking $v_h=\mathcal{L}_h(u_0)-u_0$ in the above equation and applying the
Cauchy-Schwartz inequality, we get 
\[
  \|(\mathcal{L}_h(u_0)-u_0, 0)\|_{1,h}^2
  \lesssim
  \|(u_0, u_0)\|_{1,h}^2 = \|u_0\|_{A_{h,0}}^2,
\] 
where the last step follows from definition of the norms.
The estimate \eqref{rd-p2} then follows from the triangle inequality and the
above inequality.

Now the inequality \eqref{rd-proj1} follows from \eqref{rd-p1}--\eqref{rd-p2} and 
$\rho_{A_h}\simeq h^{-2}$ (from Lemma \ref{lemma:rd1}):
\begin{align*}
  \|\Pi_h u_0\|_{A_h}\le 
  \|\Pi_h u_0-u_0\|_{A_h}+ 
  \|u_0\|_{A_h} 
  \lesssim h^{-1} \|\Pi_h u_0-u_0\|_{0,h}
  +
  \|u_0\|_{A_{h,0}}
  \lesssim 
  \|u_0\|_{A_{h,0}}.
\end{align*}

To prove the inequality \eqref{rd-proj2}, we use the following result
for the averaging operator $S_1$
\begin{align*}
  \|P_h\widehat u_h - \mathcal{L}_h(\widehat u_h)\|_{\Omega}^2
  \lesssim h^2\|\nabla \mathcal{L}_h(\widehat u_h)\|_{\Th}^2
  +\left\|\jmp{\mathcal{L}_h(\widehat u_h)}\right\|_{0,h}^2
\end{align*}
where $\jmp{\cdot}$ is the jump on interior facets.
Then the inverse inequality implies that 
\begin{align}
  \label{rd-p3}
\|P_h\widehat u_h - \mathcal{L}_h(\widehat u_h)\|_{\Omega}
\lesssim \|\mathcal{L}_h(\widehat u_h)\|_{\Omega},
\end{align}
and definition of the norm $\|\cdot\|_{A_h}$ implies that 
\begin{align}
  \label{rd-p4}
\|P_h\widehat u_h - \mathcal{L}_h(\widehat u_h)\|_{\Omega}
\lesssim h\|\widehat u_h\|_{A_h}.
\end{align}
Hence, we have 
\begin{align*}
  \|P_h \widehat u_h\|_{A_{h,0}}^2\lesssim &\;
\|P_h \widehat u_h-\mathcal{L}_h(\widehat u_h)\|_{A_{h,0}}^2
  +\|\mathcal{L}_h(\widehat u_h)\|_{A_{h,0}}^2\\
  =&\;
\|\nabla(P_h \widehat u_h-\mathcal{L}_h(\widehat u_h))\|_{\Th}^2
+\tau\|P_h \widehat u_h-\mathcal{L}_h(\widehat u_h)\|_{\Omega}^2+
  \|\mathcal{L}_h(\widehat u_h)\|_{A_{h,0}}^2\\
  \lesssim&\;
  h^{-2}\|P_h \widehat u_h-\mathcal{L}_h(\widehat u_h)\|_{\Omega}^2
  +\|\mathcal{L}_h(\widehat u_h)\|_{A_{h,0}}^2 \quad\quad\quad\text{by the inverse inequality 
  and \eqref{rd-p3},}\\
  \lesssim&\;
\|\widehat u_h\|_{A_h}^2
  \quad\quad\quad
  \quad\quad\quad
  \quad\quad\quad
  \quad\quad\quad
  \quad\quad\quad
  \quad\;\;
  \text{by the definition of norms
  and \eqref{rd-p4}}.
\end{align*}

Finally, let us prove the approximation property \eqref{rd-proj3}.
\begin{align*}
    \|\widehat u_h-\Pi_h P_h \widehat u_h\|_{0,h}\lesssim &\;
    \|
    \widehat u_h-\mathcal{L}_h(\widehat u_h)
\|_{0,h}
+
    \|
    \mathcal{L}_h(\widehat u_h)-P_h \widehat u_h
\|_{0,h}
+
    \|
    P_h \widehat u_h
    -\Pi_h P_h\widehat u_h
\|_{0,h}\\
    \lesssim &\;
    h\|
    \widehat u_h
\|_{A_h}
+
    \|
    \mathcal{L}_h(\widehat u_h)-P_h \widehat u_h
\|_{\Omega}
+
   h \|
   \nabla (P_h\widehat u_h)
\|_{\Omega} \\
    \lesssim &\;
    h\|
\widehat u_h
\|_{A_h}
+
   \|
   \mathcal{L}_h(\widehat u_h)-
   P_h\widehat u_h
\|_{\Omega} 
+
   h \|
   \nabla (\mathcal{L}_h(\widehat u_h))
\|_{\Th}\\
    \lesssim &\;
    h\|\widehat u_h\|_{A_h},
\end{align*}
where the first step follows from the triangle inequality, 
the second step follows from the inverse inequality and 
\eqref{rd-p1}, the third step follows from the triangle inequality 
and inverse inequality, and the last step follows from 
\eqref{rd-p4} and definition of the $A_h$-norm.
\end{proof}

\textbf{(v)}
Finally, 
we conclude the following optimality of the auxiliary space 
preconditioner 
by invoking Theorem \ref{thm:asp} and applying 
Lemma \ref{lemma:rd1}--Lemma \ref{lemma:rd3}.
\begin{theorem}[ASP for the condensed HDG operator \eqref{rd-ah}]
  \label{thm:rd}
  Let 
  $B_h
  =R_h+\Pi_hB_{h,0}\Pi_h^t
  $ be the auxiliary space preconditioner for the operator 
  $A_h$ in \eqref{rd-ah} with 
  $R_h$ being the Jacobi smoother for $A_h$, 
  $B_{h,0}$ being an optimal preconditioner for $A_{h,0}$ in \eqref{rd-ah0} such
  that 
  $
  \kappa (B_{h,0}A_{h,0})\simeq 1,
  $
  and $\Pi_h$ being the projector in \eqref{rd-pi}. 
  Then, 
  $
  \kappa (B_{h}A_{h})\simeq 1.
  $
\end{theorem}

\subsection{Divergence-conforming HDG for vectorial reaction-diffusion}
\label{sec:vrd}
\subsubsection{The model and the HDG scheme}
We consider the 
following constant-coefficient vectorial reaction-diffusion equation with 
a homogeneous Dirichlet 
boundary condition:
\begin{align}
  \label{vrd-eq}
  -\triangle \bld u + \tau \bld u = \bld f \;\; \text{in } \Omega, \quad\quad 
\bld u|_{\partial \Omega}=0
\end{align}
where $\tau\ge 0$ is the non-negative reaction coefficient.
Here $\bld u: \Omega \rightarrow \mathbb{R}^d$ is the solution to be
approximated.
Again, our analysis focuses on this constant coefficient case, while 
variable coefficients will be covered in the numerical experiments. 

We apply the (interior penalty based) divergence-conforming HDG scheme with projected jumps 
\cite[Chapter 2]{Lehrenfeld10}
to discretize the equation \eqref{vrd-eq}. Given a polynomial degree $k\ge 1$, 
the HDG scheme reads as follows:
find $(\bld u_h, \widehat{\bld u}_h)\in \Vh\times \Vhat$ such that 
  \begin{align}
  \label{hdg-vrd}
  \bld a_h\left((\bld u_h, \widehat{\bld u}_h), (\bld v_h, \widehat{\bld v}_h)
    \right) = \int_{\Omega} \bld f\cdot\bld v_h\mathrm{dx},\quad
    \forall (\bld v_h, \widehat{\bld v}_h)\in 
    \Vh\times \Vhat,
  \end{align}
where the bilinear form 
\begin{align*}
  \bld a_h((\uh, \uhat), (\vh, \vhat))
  := \sum_{K\in\Th}&\; \int_{K}(\nabla \uh:\nabla \vh+\tau \uh\cdot \vh)\mathrm{dx}
  -\int_{\partial K} (\nabla \uh) \bld n\cdot \mathsf{tang}(\vh-\vhat)\mathrm{ds}\\
  &\!\!\!\!\!\!\!\!\!\!\!\!\!\!\!\!\!\!
  \!\!\!\!\!\!\!\!\!\!\!\!\!\!\!\!\!\!
  \!\!\!\!\!\!\!\!\!\!\!\!\!\!\!\!\!\!
  -\int_{\partial K} (\nabla \vh)\bld n\cdot\mathsf{tang}(\uh-\uhat)\mathrm{ds}
  +\int_{\partial K}\frac{\alpha k^2}{h} 
  \mathsf{tang}(\bld P_{k-1} \uh-\uhat)\cdot\tang(\bld P_{k-1}
  \vh-\vhat)\mathrm{ds},
\end{align*}
with $\bld P_{k-1}$ denoting the $L^2$-projection into the space of
{\it vectorial} facet-piecewise
polynomials of degree $k-1$. 
Here, again, we take the stability parameter $\alpha=4$ 
to ensure  coercivity. In particular, 
we have the following norm equivalence result:
\begin{align}
  \label{hdg-vrd-norm}
  \bld a_h((\uh, \uhat), (\uh, \uhat))
  \simeq \|(\uh, \uhat)\|_{1,h}^2, \quad\forall 
(\uh, \uhat)\in \Vh\times \Vhat.
\end{align} 
where 
\[
  \|(\uh, \uhat)\|_{1,h}^2:=\sum_{K\in\Th}(\|\nabla \uh\|_K^2
  +\tau \|\uh\|_K^2
  +\frac{1}{h}\|\uh-\uhat\|_{\partial K}^2)
\]
is a norm on $\Vh\times \Vhat$.
\subsubsection{Static condensation}
Static condensation for the scheme \eqref{hdg-vrd} is
similar to that for the scalar case \eqref{hdg-rd}, but the notation is a bit
more involved.

We use the space splitting in \eqref{vfe} to perform static condensation
that locally eliminate DOFs associated with the bubble space $\bld V_h^{k,o*}$.
Denoting the following lifting operator 
$\bld{\mathcal{L}}_h:\Vhd\times \Vhat\rightarrow \Vhos$:
given $(\whd,\what) \in \Vhd\times \Vhat$, 
$\bld{\mathcal{L}}_h(\whd, \what)$ is the unique function in $\Vhos$ such that
\begin{align}
  \label{vrd-L1}
  \bld a_h((\bld{\mathcal{L}}_h(
  \whd,\what), 0), (\vho,0)) = 
  -\bld a_h((\whd, \what), (\vho,0)) 
  ,\;\;\;\; \forall \vho\in \Vhos,
\end{align}
and let $\bld{\mathcal{L}}_h^f$ be the unique function in $\Vhos$
such that 
\begin{align}
  \label{vrd-L2}
  \bld  a_h((\bld{\mathcal{L}}_h^f, 0), (\vho,0)) = 
  \int_\Omega\bld f \cdot\vho\mathrm{dx},\;\; \forall \vho\in \Vhos.
\end{align}
Again, well-posedness of the above linear systems easily follows from the coercivity result 
\eqref{hdg-vrd-co}.

\begin{lemma}[Characterization of the condensed system for \eqref{hdg-vrd}]
Let $(\uh,\uhat)\in \Vh\times \Vhat$
be the unique solution 
  to the HDG scheme \eqref{hdg-vrd}.
  Then,
  \begin{subequations}
  \begin{align}
    \label{vrd-uh}
    \uh = &\;\uhd+\bld{\mathcal{L}}_h(\uhd, \uhat)+\bld{\mathcal{L}}_h^f, 
  \end{align}
  and 
  $(\uhd, \uhat)\in 
  \Vhd\times \Vhat
  $ is the unique solution to the following condensed system
  \begin{align}
    \label{vrd-uhat}
    \bld  a_h\left(
      (\uhd+\bld{\mathcal{L}}_h(\uhd,\uhat), \uhat), 
      (\vhd+\bld{\mathcal{L}}_h(\vhd,\vhat), \vhat) 
  \right)
    = 
    -\bld a_h((\bld{\mathcal{L}}_h^f, 0), 
    (\vhd, \vhat))
    +\int_{\Omega}\bld f\cdot\vhd\,\mathrm{dx},
  \end{align}
  for all $(\vhd,\vhat)\in\Vhd\times \Vhat$.
  In particular, the reduced system \eqref{vrd-uhat} is symmetric and positive
  definite, and 
  \begin{align}
    \label{vnorm-eq}
    \bld  a_h\left(
      (\vhd+\bld{\mathcal{L}}_h(\vhd,\vhat), \uhat), 
      (\vhd+\bld{\mathcal{L}}_h(\vhd,\vhat), \vhat) 
  \right)
    \simeq 
    \|(\vhd+\bld{\mathcal{L}}_h(\vhd,\vhat), \vhat)\|_{1,h}^2,
  \end{align}
  for all 
$(\vhd, \vhat)\in \Vhd\times \Vhat$.
  \end{subequations}
\end{lemma}
\begin{proof}
  See the proof of Lemma \ref{lemma:rd}.
\end{proof}

\subsubsection{The auxiliary space preconditioner}
Now, we follow Remark \ref{rk1} to construct the auxiliary space perconditioner 
for the reduced HDG system \eqref{rd-uhat}.

\textbf{(i)} We denote the following $L^2$-like inner product on 
the compound global space $\Vhd\times \Vhat$: 
\begin{align}
  \label{vrd-l2}
\left((\uhd, \uhat), (\vhd, \vhat)\right)_{0,h}:= 
  (1+\tau h^2)
  \int_{\Omega}(\uhd\cdot\vhd)\mathrm{dx}
  +\sum_{F\in\Eh}h\int_F\tang(\uhat)\cdot\tang(\vhat)\mathrm{ds},
\end{align}
whose induced norm is denoted as 
$\|(\vhd, \vhat)\|_{0,h}^2:=\left((\vhd, \vhat), (\vhd, \vhat)\right)_{0,h}$.
We define the condensed HDG operator $\bld A_h:
\Vhd\times \Vhat\rightarrow\Vhd\times \Vhat$ as 
\begin{align}
  \label{vrd-ah}
  \left(\bld A_h(\uhd, \uhat), (\vhd, \vhat)\right)_{0,h}:=
    \bld  a_h\left(
      (\uhd+\bld{\mathcal{L}}_h(\uhd,\uhat), \uhat), 
      (\vhd+\bld{\mathcal{L}}_h(\vhd,\vhat), \vhat). 
  \right)
\end{align}
The following result gives an estimation of the 
spectral radius of $\bld A_h$ and its condition number.
\begin{lemma}[Spectral radius and condition number]
  \label{lemma:vrd1}
  Let $\bld A_h$ be the operator given in \eqref{vrd-ah}. Then,
  there holds $\rho_{\!\!\bld A_h}\simeq h^{-2}$ and 
  $\kappa (\bld A_h) \lesssim \frac{h^{-2}}{\min\{
  \frac{\tau+1}{1+\tau h^2},\, h^{-2}\}}$.
\end{lemma}
\begin{proof}
  The proof is similar to that for Lemma \ref{lemma:rd1}. 
  Here we only sketch the main steps.
  
  A simple scaling argument implies that 
  \begin{align}
    \label{vrd-1}
    \|\uhd\|_\Omega^2 \simeq\sum_{F\in\Eh} h\|\uhd\cdot \bld n\|_{F}^2, 
    \quad \forall \uhd \in \Vhd.
  \end{align}
  To simplify notation, we denote 
  \[
    \uh^+:=\uhd+\bld{\mathcal{L}}_h(\uhd, \uhat).
  \] 
  The Cauchy-Schwarz inequality implies that 
  \[
    \|(\uhd, \uhat)\|_{\bld A_h}\simeq 
    \|(\uh^+, \uhat)\|_{1,h}\lesssim
    \|(\uhd, \uhat)\|_{1,h}
    \lesssim
    h^{-1}\|(\uhd,\uhat)\|_{0,h},
  \] 
  where the last step involves the inverse inequality and triangle inequality.
 Hence, $\rho_{\!\! \bld A_h}\lesssim h^{-2}$.

To prove the condition number estimate, we again use 
the following discrete Poincar\'e inequality
\[
  \|\uh^+\|_{\Omega}\lesssim \|(\uh^+,\uhat)\|_{1,h}.
\] 
Then, 
  \begin{align*}
    \|(\uh^+, \uhat)\|_{1,h}^2
    \gtrsim&\; 
    (\tau+1)\|\uh^+\|_{\Omega}^2
    +\sum_{F\in\Eh}\frac{1}{h}\|\tang(\uh^+-\uhat)\|_F^2\\
    \gtrsim&\; 
    \sum_{F\in\Eh}
    \left(
      (\tau+1)(h\|\uhd\cdot \bld n\|_F^2+
      h\|\tang(\uh^+)\|_F^2)+\frac{1}{h}\|\tang(\uh^+-\uhat)\|_F^2
    \right)\\
    \gtrsim&\;
    (\tau +1)\|\uhd\|_{\Omega}^2
    +
    \min\{\tau+1, h^{-2}\}
    \sum_{F\in\Eh}
     h \|\tang(\uhat)\|_F^2\\
    \gtrsim&\;
    \min\{\frac{\tau+1}{1+\tau h^2}, h^{-2}\}
    \|(\uhd, \uhat)\|_{0,h},
  \end{align*}
  where in the second step we used the inverse inequality 
  and the fact that $\uh^+\cdot \bld n = \uhd\cdot\bld n$.
  This implies 
 $
  \lambda_{\!\!\bld A_h}^{\min}
  \gtrsim  
  \min\{\frac{\tau+1}{1+\tau h^2}, h^{-2}\}
  $, hence
  \[
    \kappa (\bld A_h) = \frac{\rho_{\!\!\bld A_h}}{\lambda_{\!\!\bld A_h}^{\min}}
    \lesssim
\frac{h^{-2}}{
\min\{\frac{\tau+1}{1+\tau h^2}, h^{-2}\}}.
  \]

  Finally, the lower bound on the spectral radius $h^{-2}\lesssim \rho_{\!\!\bld A_h}$
  following from the following estimate
  \begin{align}
    \label{rd-scale}
    \|(\bld \phi_h^\partial, \widehat{\bld \phi}_h)\|_{\bld A_h}^2\simeq
  h^{-2}\|(\bld\phi_h^{\partial_h}, \widehat{\bld\phi}_h)\|_{0,h}^2,
  \end{align}
  for any pair of 
  local bases functions 
  $(\bld \phi_h^\partial, \widehat{\bld \phi}_h)\in \Vhd\times \Vhat$, 
  which is obtained from 
  norm equivalence of finite dimensional spaces and the scaling argument.
\end{proof}
\begin{remark}[The case with a large reaction parameter]
  \label{rk:3.3}
When the reaction parameter $\tau\gtrsim h^{-2}$, 
Lemma \ref{lemma:vrd1} implies that $\kappa(\bld A_h)\simeq 1$.
In this case, a simple Jacobi preconditioner is again
an optimal one for $\bld A_h$.
\end{remark}

\textbf{(ii)} 
Similar to the scalar case, we can show that 
the simple Jacobi smoother satisfies \eqref{asp-as1}, and 
in practical implementation, we prefer to use a block symmetric Gauss-Seidel
smoother to improve its efficiency.

\begin{lemma}[Spectum of the Jacobi preconditioner]
  \label{lemma:vrd2}
  Let   $\bld D_h$ be the  diagonal component of 
  $\bld A_h$.
  Then, the following equivalence holds
  for the Jacobi smoother $\bld R_h=\bld D_h^{-1}$:
  \[
    \left(\bld R_h(\uhd,\uhat), (\uhd, \uhat)\right)_{0,h}\simeq \rho_{\!\!\bld A_h}^{-1}
    \|(\uhd, \uhat)\|_{0,h}^2.
  \] 
\end{lemma}
\begin{proof}
  See the proof of Lemma \ref{lemma:rd1}.
\end{proof}

\textbf{(iii)} 
The auxiliary space under consideration is the (vectorial) continuous piecewise linear
finite element space $\Vc$ in \eqref{vh10}.
The associated bilinear form $\bld a_{h,0}$ and the linear operator 
$\bld A_{h,0}$
are given as follows:
\begin{align}
  \label{vrd-ah0}
  [\bld A_{h,0}\bld u_0, \bld v_0] = \bld a_{h,0}(
  \bld u_0, \bld v_0) := \int_{\Omega}(\nabla \bld u_0:\nabla\bld v_0+\tau
 \bld u_0\cdot \bld v_0)\mathrm{dx},\quad\quad
  \forall \bld u_0,\bld v_0\in \Vc,
\end{align}
where $[\cdot,\cdot]$ is the usual $L^2$-norm on $\Vc$.
A robust preconditioner $\bld B_{h,0}$ for the operator $\bld A_{h,0}$
can be readily obtained from a vector version of that for the scalar case 
as the operator \eqref{vrd-ah0} is decoupled for each component.
Here we again take $\bld B_{h,0}$ to be hypre's BoomerAMG preconditioner \cite{hypre,
Henson02}.

\textbf{(iv)} We take the operator $\underline{\bld \Pi_h}=(\bld \Pi_h^\partial, 
\widehat{\bld
\Pi}_h):\Vc\rightarrow \Vhd\times \Vhat$ 
to be the following projection:
\begin{subequations}
  \label{vrd-pi}
\begin{align}
  \label{vrd-pi1}
  \sum_{F\in\Eh}
  \int_F
  (\bld\Pi_h^\partial \bld u_o\cdot \bld n)
  (\vhd\cdot\bld n)\mathrm{ds}
  = 
  \sum_{F\in\Eh}
  \int_F
  (\bld u_o\cdot \bld n)
  (\vhd\cdot\bld n)\mathrm{ds},\quad
  \forall \vhd\in \Vhd,\\
  \label{vrd-pi2}
  \sum_{F\in\Eh}
  \int_F
 \tang(\widehat{\bld\Pi}_h \bld u_o)\cdot
 \tang(\vhat)
  \mathrm{ds}
  = 
  \sum_{F\in\Eh}
  \int_F
 \tang(\bld u_o)\cdot
 \tang(\vhat)
\mathrm{ds},\quad
  \forall \vhat\in \Vhat,
\end{align}
\end{subequations}
and the operator $\bld P_h:\Vhd\times\Vhat\rightarrow \Vc$ 
to be 
\begin{align}
 \label{vrd-p}
  \bld P_h(\uhd, \uhat) =S_1(\uh^+),
\end{align}
where recall that $\uh^+:=\uhd+\bld{\mathcal{L}}_h(\uhd, \uhat)$
and $S_1$ is the simple nodal 
average given in \eqref{s1}. 

The following result shows that $\underline{\bld \Pi_h}$ and $\bld P_h$ satisfies both the stability 
and approximation properties needed to apply the ASP theory, Theorem
\ref{thm:asp}.
\begin{lemma}[Properties of $\underline{\bld \Pi_h}$ and $\bld P_h$]
  \label{lemma:vrd3}
  Let the operator $\underline{\bld\Pi_h}$ be given in \eqref{vrd-pi} and 
  $\bld P_h$ be given in \eqref{vrd-p}. Then
  the following inequalities holds
\begin{subequations}
  \label{vrd-proj}
  \begin{align}
  \label{vrd-proj1}
  \|\underline{\bld \Pi_h}\bld u_0\|_{\bld A_h}\lesssim &\;\|\bld u_0\|_{\bld A_{h,0}},\\
  \label{vrd-proj2}
  \|\bld P_h(\uhd, \uhat)\|_{\bld A_{h,0}}\lesssim &\;\|(\uhd, \uhat)
 \|_{\bld A_{h}},\\
  \label{vrd-proj3}
  \|(\uhd,\uhat)-\underline{\bld \Pi_h} \bld P_h (\uhd,\uhat)
 \|_{0,h}\lesssim &\;
    \rho_{\!\!\bld A_h}^{-1/2}
    \|(\uhd,\uhat)\|_{\bld A_{h}}.
  \end{align}
\end{subequations}
\end{lemma}
\begin{proof}
  The proof is similar to that for Lemma \ref{lemma:rd3}. Here we only sketch
  the main steps.

  Since $\Vc\subset \Vhd$, we have $\bld \Pi_h^\partial \bld u_0 = \bld u_0$
  for any $\bld u_0\in\Vc$.
Hence,
  \begin{align}
    \label{vrd-p1}
    \|\underline{\bld \Pi_h}\bld u_0-(\bld u_0, \bld u_0)\|_{0,h}
    =
    \|(0, \widehat{\bld \Pi}_h\bld u_0-\bld u_0)\|_{0,h}
    \lesssim
    h\|\nabla \bld u_0\|_{\Omega}.
  \end{align} 
  Next, we can prove 
 \begin{align}
   \label{vrd-p2}
   \|(\bld u_0, \bld u_0)\|_{\bld A_h}\lesssim \|\bld u_0\|_{\bld A_{h,0}},\quad \forall
   \bld u_0\in \Vc.
 \end{align}
Then the inequality \eqref{vrd-proj1} follows from
the above two estimates and $\rho_{\!\bld A_h}\simeq h^{-2}$.

To prove the inequality \eqref{vrd-proj2}, we again use the following result
for the average operator $S_1$
\begin{align*}
  \|\bld P_h(\uhd, \uhat) - \uh^+\|_{\Omega}^2
  \lesssim h^2\|\nabla \uh^+\|_{\Th}^2
  +\left\|\jmp{\uh^+}\right\|_{0,h}^2,
\end{align*}
which implies that 
\begin{align}
  \label{vrd-p3}
  \|\bld P_h(\uhd,\uhat) - \uh^+\|_{\Omega}
  \lesssim \min\left\{\|\uh^+\|_{\Omega},\;\; h\|(\uhd,\uhat)\|_{\bld A_h}\right\}.
\end{align}
Hence,  we have 
\begin{align*}
  \|\bld P_h (\uhd,\uhat)\|_{\bld A_{h,0}}\lesssim &\;
\|\bld P_h(\uhd, \uhat)-\uh^+\|_{\bld A_{h,0}}
  +\|\uh^+\|_{\bld A_{h,0}}
  \lesssim
  \|(\uhd, \uhat)\|_{\bld A_h}.
\end{align*}

Finally, let us prove the approximation property \eqref{rd-proj3}.
The further simply notation, we denote 
$\bld P_h^+:=\bld P_h(\uhd,\uhat)$.
Then, there holds
\begin{align*}
    \|
    (\uhd,\uhat)-\underline{\bld\Pi_h}\bld P_h^+
   \|_{0,h}^2
    \lesssim &\;
    \|(\uhd-\bld P_h^+, \uhat-
    \uh^+
    )
\|_{0,h}^2
+
    \|(0, 
    \uh^+-\bld P_h^+
    )
\|_{0,h}^2
+
    \|
    (\bld P_h^+,\bld P_h^+)-\underline{\bld\Pi_h}\bld P_h^+
   \|_{0,h}^2\\
    \lesssim &\;
    (1+\tau h^2)\|\uhd-\bld P_h^+\|_{\Omega}^2
    +h^{2} \|(\uh^+, \uhat)\|_{\bld A_h}^2
+
    \| 
    \uh^+-\bld P_h^+
\|_{\Omega}^2
+
h^2\|\nabla \bld P_h^+\|_{\Omega}^2,\\
    \lesssim &\;
    (1+\tau h^2)\|\uh^+-\bld P_h^+\|_{\Omega}^2
    +h^{2} \|(\uh^+, \uhat)\|_{\bld A_h}^2
    \lesssim 
    h^{2} \|(\uh^+, \uhat)\|_{\bld A_h}^2,
  \end{align*}
where 
in the third step we used the fact that 
\[
\|\uhd-\bld P_h^+\|_\Omega^2
\simeq
h\sum_{F\in\Eh}\|(\uhd-\bld P_h^+)\cdot \bld n\|_F^2
= 
h\sum_{F\in\Eh}\|(\uh^+-\bld P_h^+)\cdot \bld n\|_F^2
\lesssim 
\|\uh^+-\bld P_h^+\|_{\Omega}^2.
\] 
This completes the sketchy proof.
\end{proof}

\textbf{(v)}
Finally, 
we conclude the following optimality of an auxiliary space 
preconditioner 
by invoking Theorem \ref{thm:asp} and applying 
Lemma \ref{lemma:vrd1}--Lemma \ref{lemma:vrd3}.
\begin{theorem}[ASP for the condensed HDG operator \eqref{vrd-ah}]
  \label{thm:vrd}
  Let 
  $\bld B_h
  =\bld R_h+\underline{\bld \Pi_h}\bld B_{h,0}
  \underline{\bld \Pi_h}^t
  $ be the auxiliary space preconditioner for the operator 
  $\bld A_h$ in \eqref{vrd-ah} with 
  $\bld R_h$ being the Jacobi smoother for $\bld A_h$, 
  $\bld B_{h,0}$ being an optimal preconditioner for $\bld A_{h,0}$ in \eqref{vrd-ah0} such
  that 
  $
  \kappa (\bld B_{h,0}\bld A_{h,0})\simeq 1,
  $
  and $\underline{\bld \Pi_h}$ being the projector in \eqref{vrd-pi}. 
  Then, 
  $
  \kappa (\bld B_{h}\bld A_{h})\simeq 1.
  $
\end{theorem}
\begin{remark}[On reduced divergence-conforming space]
 All  the above results still 
 hold if we replace the divergence-conforming space $\Vh$ in 
 the HDG scheme
 \eqref{hdg-vrd} by the reduced divergence-conforming space 
 $\Vhr$ in \eqref{vl2vc0}. The only change would be a smaller local space 
 $\Vho$ compared with the original local space $\Vhos$ for the scheme
 \eqref{hdg-vrd}.
 While such a modification leads to a smaller set of total DOFs, it  is not suggested for the HDG discretization of the
 current reaction-diffusion equation \eqref{vrd-eq}
 as it leads to accuracy loss compared with the original 
 HDG scheme \eqref{hdg-vrd}.
 On the other hand,
 the reduced system would serve as a good  preconditioner
 for the velocity block of a divergence-free HDG discretization of
 incompressible flow.
 It can also be used to precondition a 
 $C^0$-continuous interior penalty HDG scheme for 
 the generalized biharmonic equation as we show next.
\end{remark}

\subsection{$C^0$-continuous interior penalty HDG for the generalized biharmonic problem}
\label{sec:bh}
\subsubsection{The model and the HDG scheme} We consider the following
constant-coefficient generalized biharmonic equation in a convex polygonal
domain 
$\Omega \subset \mathbb{R}^2$ with a simply supported boundary condition:
\begin{align}
  \label{bh-eq}
  -\triangle^2 \phi - \tau \triangle \phi = f \;\; \text{in } \Omega,
\quad  \;\;\phi|_{\partial \Omega}=\triangle\phi|_{\partial \Omega} = 0.
\end{align}
Calculus identities show that 
\[
  \triangle^2 \phi = \nabla\times(\nabla \cdot \nabla (\vv{\nabla}
  \times\phi)),\quad
  \triangle \phi = \nabla\times (\vv{\nabla}\times \phi),
\] 
where $\vv{\nabla}\times \phi :=(\partial_y\phi, -\partial_x\phi)$
is the {\it vectorial} curl operator (rotated gradient) for a scalar field $\phi$ and 
 ${\nabla}\times (v_1, v_2) :=\partial_x v_2-\partial_y v_1$
 is the {\it scalar} curl operator (rotated divergence) for the vector field
 $(v_1,v_2)$.
 Also, for any function $\phi$ that vanishes on the boundary $\partial \Omega$,
 there holds
 \[
   \triangle \phi|_{\partial\Omega}
 = \partial_{n}\partial_{n}\phi|_{\partial\Omega}
   = 
   \left(((\nabla \vv{\nabla}\times \phi)\bld n)\cdot\bld
   t\right)|_{\partial\Omega} = 0,
 \] 
 where $\partial_n$ is the directional derivative along the normal direction
 $\bld n=(n_1, n_2)$, 
and $\bld t = (n_2, -n_1)$ is the tangential direction.
Hence, the equation \eqref{bh-eq} is identical to the following form:
\begin{align}
  \label{bh-eq1}
  -\nabla\times(\nabla \cdot \nabla (\vv{\nabla}
  \times\phi))
  -\tau \nabla\times (\vv{\nabla}\times \phi)=f
  \;\; \text{in } \Omega,
\quad  \;\;\phi|_{\partial \Omega}=
   \left(((\nabla \vv{\nabla}\times \phi)\bld n)\cdot\bld
   t\right)|_{\partial\Omega} = 0.
\end{align}
Next, we present a $C^0$-continuous interior penalty HDG (CIP-HDG) discretization of the
equation \eqref{bh-eq1}. We didn't find in the literature on publications on 
CIP-HDG formulations for the generalized biharmonic equation, but it can be
easily adapted from a known CIP-DG scheme \cite{Brenner05}.
We mention that the implementation of a version of the CIP-HDG scheme for the
biharmonic equation was already used in unit 2.9 of the {\it i-tutorials}
of the NGSolve software \cite{ngsolve}.

Given a polynomial degree $k\ge 1$, 
the CIP-HDG scheme with projected jumps reads as follows:
find $(\phi_h, \widehat{\bld u}_h)\in \Sh\times \Vhato$ such that 
  \begin{align}
  \label{hdg-bh}
  k_h\left((\phi_h, \widehat{\bld u}_h), (\psi_h, \widehat{\bld v}_h)
    \right) = \int_{\Omega} f\,\psi_h\mathrm{dx},\quad
    \forall (\psi_h, \widehat{\bld v}_h)\in 
    \Sh\times \Vhato,
  \end{align}
where the bilinear form 
\begin{align*}
  k_h((\phi_h, \uhat), (\psi_h, \vhat))
  := \sum_{K\in\Th}&\; \int_{K}(\nabla \vph:\nabla \vqh+\tau \vph\cdot
  \vqh)\mathrm{dx}\\
                   &\;
  -\int_{\partial K} (\nabla \vph) \bld n\cdot \mathsf{tang}(\vqh-\vhat)\mathrm{ds}\\
  &\;
  -\int_{\partial K} (\nabla \vqh)\bld n\cdot\mathsf{tang}(\vph-\uhat)\mathrm{ds}
\\&
\;+\int_{\partial K}\frac{\alpha k^2}{h} 
  \mathsf{tang}(\bld P_{k-1} \vph-\uhat)\cdot\tang(\bld P_{k-1}
  \vqh-\vhat)\mathrm{ds},
\end{align*}
where $\uhat\in \Vhato$ is the approximation of the tangential component of 
$\vph$ on the mesh skeleton.
Note that we do not impose any boundary constraints on the tangential facet finite
element space $\Vhato$ to respect the simply supported boundary condition in
\eqref{bh-eq1}. 
Here, again, we take the stability parameter $\alpha=4$ 
to ensure the coercivity. In particular, the following norm equivalence holds
\begin{align}
\label{hdg-bh-norm}
  k_h((\ph, \uhat),(\ph, \uhat))\simeq
  \|(\ph, \uhat)\|_{1,h}^2, 
\end{align} 
where 
\[
  \|(\ph, \uhat)\|_{1,h}^2:=\sum_{K\in\Th}(\|\nabla \vph\|_K^2
  +\tau \|\vph\|_K^2
  +\frac{1}{h}\|\vph-\uhat\|_{\partial K}^2)
\]
is a norm on $\Sh\times \Vhato$.
A direct comparison of the CIP-HDG bilinear form $k_h$ in \eqref{hdg-bh}
with the divergence-conforming HDG bilinear form $\bld a_h$
in \eqref{hdg-vrd} implies that 
\begin{align}
  \label{form-eq}
  k_h((\ph, \uhat),(\qh, \vhat))
  =
  \bld a_h((\vph, \uhat),(\vqh, \uhat)),
\end{align}
for all 
  $(\ph, \uhat), (\qh, \vhat)\in \Sh\times \Vhato$.

\subsubsection{The auxiliary space preconditioner}
Here we directly construct the auxiliary space preconditioner
for the system \eqref{hdg-bh} without static condensation.
The analysis with static condensation is similar the previous 
subsection and we omit it to avoid unnecessary repetition.
In our numerical implementation, static condensation is of course activated 
to improve efficiency of the solver.
The key idea is to use the  divergence-conforming HDG scheme 
\eqref{hdg-vrd} for
reaction-diffusion to precondition the scheme \eqref{hdg-bh}, see the original
idea in \cite{Ayuso14}. 
The obtained preconditioner is actually a fictitious space preconditioner
\cite{N91} as no smoother is involved in the preconditioner.

We define the HDG operator $K_h:
\Sh\times \Vhato\rightarrow\Sh\times \Vhato$ as 
\begin{align}
  \label{bh-ah}
  \left[K_h(\ph, \uhat), (\qh, \vhat)\right]_{0,h}:=
    k_h\left(
      (\ph, \uhat), 
      (\qh
     , \vhat) 
  \right),
\end{align}
where $[\cdot,\cdot]_{0,h}$ is an inner product on $\Sh\times \Vhato$.
The auxiliary space to be considered is 
$
\Vhr\times \Vhato.
$ 
With an abuse of notation, we denote 
$\bld A_h:\Vhr\times \Vhato\rightarrow \Vhr\times \Vhato$ as the linear operator
associated with the bilinear form $\bld a_h$ in \eqref{hdg-vrd}:
\begin{align}
  \label{ah-2}
  \left(\bld A_h(\uh, \uhat), (\vh, \vhat)\right)_{0,h}:=
    \bld  a_h\left(
      (\uh, \uhat), 
      (\vh, \vhat)
    \right)
      ,\quad 
      \forall 
      (\uh, \uhat), 
      (\vh, \vhat)\in \Vhr\times \Vhato 
\end{align}
An optimal preconditioner for $\bld A_h$ was already constructed in Theorem 
\ref{thm:vrd}.
We further define the following mapping operator 
$\underline{\mathsf{\Pi}_h}=(\mathsf{\Pi}_h, \widehat{\mathsf{\Pi}}_h):\Vhr\times \Vhato\rightarrow
\Sh\times \Vhato$:
\begin{subequations}
  \label{bh-pi}
\begin{align}
  \label{bh-pi1}
\int_{\Omega}
(\vv{\nabla}\times
\mathsf{\Pi}_h(\uh, \uhat) \cdot 
\vqh
  \mathrm{dx}
  = 
\int_{\Omega}
(\uh \cdot 
\vqh)
  \mathrm{dx}
  ,\quad
  \forall \qh\in \Sh,\\
  \label{bh-pi2}
  \sum_{K\in\Th}
  \int_{\partial K}
  \tang(\widehat{\mathsf{\Pi}}_h(\uh, \uhat)
  -{\mathsf{\Pi}}_h(\uh, \uhat)
  )\cdot
 \tang(\vhat)
  \mathrm{ds}
  =0, \quad 
  \forall \vhat\in \Vhato.
\end{align}
\end{subequations}
Here \eqref{bh-pi2} implies that 
\begin{align}
  \label{phat}
  \tang(\widehat{\mathsf{\Pi}}_h(\uh, \uhat))
  =\tang(\bld P_{k-1}\avg{\mathsf{\Pi}_h(\uh, \uhat)}),
\end{align}
where $\avg{\mu}$ is the average operator on interior facets of a function
$\mu$,
and $\avg{\mu}|_{\partial \Omega} = \mu$.

The main result of this subsection is now summarized below.
\begin{theorem}[ASP for the HDG operator \eqref{bh-ah}]
  \label{thm:bh}
  Let $M_h= \underline{\mathsf{\Pi}_h}\bld B_{h}
  \underline{\mathsf{\Pi}_h}^t
  $ 
  be the auxiliary space preconditioner for the operator $K_h$ in
  \eqref{bh-ah} with $\underline{\mathsf{\Pi}_h}$ given in \eqref{bh-pi}
  and $B_h$ being an optimal preconditioner for $A_h$ in \eqref{ah-2}.
  Then, $\kappa(M_hK_h)\simeq 1$.
\end{theorem}
\begin{proof}
  Taking $\mathsf{P}_h:\Sh\times \Vhato\rightarrow \Vhr\times \Vhato$ as the following
  natural inclusion:
  \[
    \mathsf{P}_h(\ph, \uhat) := (\vph, \uhat),\quad \forall (\ph, \uhat)\in
\Sh\times \Vhato.
  \]
  It is trivial to show that 
  \[
    \|\mathsf{P}_h(\ph, \uhat)\|_{\bld A_h} =
    \|(\ph, \uhat)\|_{K_h},
  \] 
  and that
  $\mathsf{P}_h$ is a right inverse of 
  $\underline{\mathsf{\Pi}_h}$:
  \[
    \underline{\mathsf{\Pi}_h}\mathsf{P}_h(\ph, \uhat)
    =(\ph, \uhat).
  \]
  Hence, we only need to prove the following
  stability of the  projector $\underline{\mathsf{\Pi}_h}$:
  \begin{align}
    \label{stab-p}
    \|
    \underline{\mathsf{\Pi}_h}(\uh, \uhat)
    \|_{K_h}\lesssim 
    \|(\uh, \uhat)
    \|_{\bld A_h}
  \end{align}
  for the optimality of the preconditioner $M_h$.
  To simplify notation, we denote 
  $\uh^*:=\vv{\nabla}\times{\mathsf{\Pi}_h}(\uh, \uhat)$.
  Hence, $$\tang(\widehat{\mathsf{\Pi}}_h(\uh, \uhat))=
  \tang(\bld P_{k-1}\avg{\uh^*})$$ by \eqref{phat}.
  We then have, by \eqref{hdg-bh-norm}, 
  \begin{align}
    \label{ss}
    \|
    \underline{\mathsf{\Pi}_h}(\uh, \uhat)
    \|_{K_h}\lesssim
  \sum_{K\in\Th}(\|\nabla \uh^*\|_K^2
  +\tau \|\uh^*\|_K^2
  +\frac{1}{h}\|\tang(\jmp{\uh^*})\|_{\partial K}^2),
  \end{align}
  where the jump $\jmp{\cdot}$ on the boundary facet $F\in\partial\Omega$ 
  is set to be {\it zero}.
  By Lemma 5.1 of \cite{Ayuso14}, we have
  \begin{align}
    \label{xxx}
    \sum_{K\in\Th}(\|\nabla (\uh^*-\uh)\|_K^2
  +\frac{1}{h}\|\tang(\jmp{\uh^*-\uh})\|_{\partial K}^2)\lesssim \|\nabla
  \cdot\uh\|_{\Omega}.
  \end{align}
  Taking $\vqh = \uh^*$ in \eqref{bh-pi1}
  and applying the Cauchy-Schwarz inequality, we further have 
  \[
    \|\uh^*\|_{\Omega}\le \|\uh\|_{\Omega}.
  \] 
  Combining the above two estimates with the inequality \eqref{ss}
  and a triangle inequality, we conclude
the  proof for   the stability result \eqref{stab-p}.
\end{proof}
\begin{remark}[On boundary conditions]
Here we remark that the assumption on the simply supported boundary condition
\eqref{bh-eq1} is crucial to our
analysis, as the result \eqref{xxx} no longer holds true if we have 
a clamped boundary condition, 
$\phi|_{\partial\Omega}=
\partial_n\phi|_{\partial \Omega}=0$.
In this case, 
we shall use $\Vhat$  as 
the facet finite element space in the CIP-HDG scheme \eqref{hdg-bh}
to respect the clamped boundary condition.
Then, we have 
$$
\tang(\widehat{\mathsf{\Pi}}_h(\uh, \uhat))|_{\partial\Omega}
  =0\not =
  \tang(\bld P_{k-1}{\mathsf{\Pi}}_h(\uh, \uhat))|_{\partial\Omega}
$$ due to the homogeneous boundary condition on $\Vhat$.
Hence the right hand side of the inequality \eqref{ss}
shall also include the boundary contribution
$
  \sum_{F\in\Eh^\partial}\frac1h\|\tang(\uh^*)\|_F^2,
$ 
which, however, can not be controlled by $\|(\uh,\uhat)\|_{\bld A_h}^2$, 
see the proof of \cite[Lemma 5.1]{Ayuso14}.

Our numerical results also indicate that Theorem \ref{thm:bh} 
fails to hold for the clamped boundary condition case as the
preconditioner is no longer optimal with respect to the mesh size $h$.
This is to be contrasted to the results in the 
previous two subsections, where we considered 
the homogeneous Dirichlet boundary condition only for the purpose of simplicity
of the presentation as the main results therein still  hold if we replace 
the Dirichlet boundary condition by other standard boundary conditions.

As can be seen clearly from the analysis of Theorem \ref{thm:bh}, we only need
to prove stability of the projector \eqref{stab-p} for the optimality of the
associated preconditioner. However, we are not able to construct an
easy-to-compute projector that achieve this. In a forthcoming paper, we will address
the preconditioning issue of \eqref{hdg-bh} with a clamped boundary condition
using a different technique.
\end{remark}

\begin{remark}[On computational cost of the preconditioner $M_h$]
The main computation cost to computer the projector 
$\underline{\mathsf{\Pi}_h}$
is the  Poisson solver \eqref{bh-pi1} on the space $\Sh$.
If we take $\bld B_h$ to be the ASP developed in Theorem \ref{thm:vrd} for 
$\bld A_{h}$, then it involves a Poisson-like solver for the 
operator $\bld A_{h,0}$ \eqref{vrd-ah0} on the space $\Vc$, which is equivalent
to two scalar Poisson-like solvers. 
Hence, an application of 
the preconditioner $M_h$ would involve two Poisson solvers on 
the high-order $H^1$-conforming space
$\Sh$ (one for 
$\underline{\mathsf{\Pi}_h}$
and one for $\underline{\mathsf{\Pi}_h}^t$) and two Poisson-like solver on the
low-order $H^1$-conforming space $\Wc$.

On the other hand, if $\tau\ge h^{-2}$, we can take 
$\bld B_h$ as the simple Jacobi preconditioner for $\bld A_h$ which is optimal
due to Lemma \ref{lemma:vrd1}. In this case, 
the cost of an application of the preconditioner $M_h$ would be two Poisson solvers on 
the high-order $H^1$-conforming space
$\Sh$.
\end{remark}

\begin{remark}[Equivalence of CIP-HDG for biharmonic equation with 
  divergence-free HDG for Stokes flow]
  It is well-known that the CIP-DG scheme \cite{Brenner05} for biharmonic equation
is equivalent to a divergence-free DG scheme
for Stokes flow, see \cite{Kanschat14}.
The same conclusion can also be made for the current 
CIP-HDG scheme \eqref{hdg-bh} and the 
divergence-free HDG scheme for a generalized Stokes flow \cite[Chapter 2]{Lehrenfeld10}.
Hence, we also obtained a robust preconditioner for 
the divergence-free HDG scheme for the generalized Stokes flow with a slip
boundary condition.
\end{remark}

\section{Numerical results}
\label{sec:num}
In this section we perform some numerical experiments for the preconditioners
discussed in the previous section.
All numerical examples are performed using the NGSolve software \cite{ngsolve}.
Full code examples are available at {\url{www.github.com/gridfunction/asp-hdg/}}.

For the symmetric interior penalty HDG scheme \eqref{hdg-rd}, we take the domain 
to be a unit cube $\Omega^{\mathsf{3d}}=[0,1]^3$ with two subdomains 
$\Omega_1^{\mathsf{3d}} = [0.25,0.5]^3\cup [0.5,0.75]^3$ and 
$\Omega_2^{\mathsf{3d}} = \Omega^{\mathsf{3d}}\backslash\Omega_1^{\mathsf{3d}}$;
for the divergence-conforming HDG scheme \eqref{hdg-vrd} and 
the CIP-HDG scheme \eqref{hdg-bh}, we take the domain to be a unit square
$\Omega^{\mathsf{2d}}=[0,1]^2$ with two subdomains 
$\Omega_1^{\mathsf{2d}} = [0.25,0.5]^2\cup [0.5,0.75]^2$ and 
$\Omega_2^{\mathsf{2d}} = \Omega^{\mathsf{2d}}\backslash\Omega_1^{\mathsf{2d}}$;
In all the examples, we take the parameter $\tau=\tau_1$ on subdomain $\Omega_1$
and $\tau= \tau_2$ on subdomain $\Omega_2$
with the constants $\tau_1,\tau_2\in\{1, 10^4\}$.
We take the 
right hand side $f=1$, and report the number of PCG iterations
need to reduce the relative residual by a factor of $10^{-10}$.
We use a regular simplicial mesh by first dividing the domain into 
uniform $N\times N$ squares in 2D or $N\times N\times N$ cubes in 3D, then 
split the square into two triangles or the cube into 6 tetrahedra.
$N=8$ for the coarsest mesh under consideration.
We take polynomial degree $k\in\{1,4,7,10\}$.

\newcommand{\jasp}{$\mathsf{JAC\text{-}ASP}$}
\newcommand{\gasp}{$\mathsf{BGS\text{-}ASP}$}
\newcommand{\dasp}{$\mathsf{ASP\text{-}DIR}$}
\newcommand{\aasp}{$\mathsf{ASP\text{-}ASP}$}

For the scalar and vectorial reaction diffusion problems in \eqref{hdg-rd} and 
\eqref{hdg-vrd}, 
we focus on the ASP in Theorems \ref{thm:rd}--\ref{thm:vrd} with a point Jacobi smoother, 
abbreviated as $\mathsf{JAC\text{-}ASP}$, and the ASP with a block symmetric 
Gauss-Seidel smoother using an {\it overlapping} 
facet-patch based block, abbreviated as $\mathsf{BGS\text{-}ASP}$.
For the block preconditioner, the total number of blocks is the total
number of (interior) facets on the mesh, with each block being associated with a 
facet $F$ that contains 
the global DOFs which are adjacent to $F$, i.e., it 
contains the collection of all global DOFs on the facet $F'$
such that $F'$ and $F$ are shared by a common simplicial element $K$.
For the generalized biharmonic problem \eqref{hdg-bh}, we focus on 
the ASP in Theorem \ref{thm:bh} with a either 
direct solver $\bld B_h$, which is abbreviate as 
$\mathsf{ASP\text{-}DIR}$,
or the ASP with a block symmetric Gauss-Seidel smoother 
in Theorem \ref{thm:vrd}, which is abbreviated as 
$\mathsf{ASP\text{-}ASP}$. 

The iteration counts 
for the HDG scheme \eqref{hdg-rd} are recorded in Table \ref{tab:1},
those for the divergence-conforming HDG scheme \eqref{hdg-vrd} are recorded in Table \ref{tab:2},
and for the CIP-HDG scheme \eqref{hdg-bh} in Table \ref{tab:3} for
the simply supported boundary case and in Table \ref{tab:4} for the clamped
boundary case.

From Table \ref{tab:1}, we observe that the iteration counts are essentially
independent of the mesh size $h=1/N$ for a fixed polynomial degree and reaction
parameter $\tau$ for both 
{\jasp} and \gasp, which verify the result of 
Theorem \ref{thm:rd}.
The iteration counts are also quite robust with respect to the variations in
$\tau$, where the case $\tau_1=\tau_2=10^4$ records the smallest number of
iterations for all tests. Due to a large reaction coefficient 
in this case,  the HDG operator \eqref{rd-ah} itself is well-conditioned, see
Remark \ref{rk:3.1}.
Moreover,  the growth of iteration counts on polynomial degree {\jasp}
seems to be linear, while that for {\gasp} is very mild where they are  almost 
independent of the polynomial degree.

The results in Table \ref{tab:2} for the divergence-conforming HDG scheme 
\eqref{hdg-vrd} are similar to
those for Table \ref{tab:1}, which numerically verify the result of 
Theorem \ref{thm:vrd}.

Finally, we also obtain similar results in Table \ref{tab:3} for the simply
supported boundary condition case, where 
we notice that switching from a direct solver in {\dasp} for the auxiliary 
operator $\bld A_h$ to the ASP in {\aasp}
only leads to a small increase on the iteration counts.
On the other hand, the results in Table \ref{tab:4} for the clamped boundary
condition case show a mesh dependency for iterations counts, especially for the
case with $\tau_1=\tau_2=1$. This suggest that the proposed preconditioner is
not robust for the clamped boundary condition case.

\begin{table}[ht]
\centering 
\resizebox{0.9\columnwidth}{!}
{%
\begin{tabular}{cc| cc|cc|cc|cccc}
\toprule
$k$& $N$  
   &\multicolumn{2}{c|}{
  $\tau_1=1, \tau_2=1$}
  &
  \multicolumn{2}{c|}{
  $\tau_1=1, \tau_2=10^4$}
  &
  \multicolumn{2}{c|}{
  $\tau_1=10^4, \tau_2=1$}
  &
  \multicolumn{2}{c}{
  $\tau_1=10^4, \tau_2=10^4$}
\tabularnewline
\midrule
  & 
  &\jasp &\gasp
  &\jasp &\gasp
  &\jasp &\gasp
  &\jasp &\gasp
  \\
\midrule
\multirow{4}{2mm}{1} 
    & 8&54  & 15&57  & 17&34  & 11&12  & 8 \\   
   &  16&55  & 17&58  & 19&50  & 16&19  & 10 \\ 
   &  32&53  & 17&59  & 19&58  & 16&41  & 11  \\
   &  64&51  & 16&55  & 17&55  & 17&46  & 12  \\
 \midrule
\multirow{4}{2mm}{4} 
    & 8&94  & 20&99  & 21&78  & 16&43  & 12  \\   
   &  16&91  & 20&98  & 21&91  & 19&58  & 13  \\ 
   &  32&86  & 20&95  & 20&93  & 19&69  & 15  \\
 \midrule
\multirow{2}{2mm}{7} 
    & 8 &142  & 23&149  & 24&124  & 19&75  & 15  \\   
   &  16&137  & 23&148  & 23&138  & 21&97  & 17  \\ 
\bottomrule
\end{tabular}}
\vspace{2ex}  
\caption{\it \textbf{3D scalar reaction-diffusion.} 
  Number of PCG iteration counts for the 
 interior penalty HDG scheme \eqref{hdg-rd}.} 
\label{tab:1} 
\end{table}

\begin{table}[ht]
\centering 
\resizebox{0.9\columnwidth}{!}
{%
\begin{tabular}{cc| cc|cc|cc|cccc}
\toprule
$k$& $N$  
   &\multicolumn{2}{c|}{
  $\tau_1=1, \tau_2=1$}
  &
  \multicolumn{2}{c|}{
  $\tau_1=1, \tau_2=10^4$}
  &
  \multicolumn{2}{c|}{
  $\tau_1=10^4, \tau_2=1$}
  &
  \multicolumn{2}{c}{
  $\tau_1=10^4, \tau_2=10^4$}
\tabularnewline
\midrule
  & 
 &\jasp
  &\gasp
  &\jasp &\gasp
  &\jasp &\gasp
  &\jasp &\gasp
   \\
\midrule
\multirow{4}{2mm}{1} 
   & 8  & 62& 15& 67&15& 52&16&34 &  10\\   
   &  16& 66& 16& 76&17& 62&19&39 &  9  \\ 
   &  32& 64& 15& 76&17& 69&18&46 &  8   \\
   &  64& 61& 15& 74&17& 74&17&56 &  11  \\
 \midrule
\multirow{4}{2mm}{4} 
   & 8  &123&19&127&20&107&18 &72  & 13 \\   
   &  16&121&19&131&20&128&22 &95  & 15  \\ 
   &  32&116&19&128&20&127&22 &109 & 16   \\
   &  64&110&19&123&20&124&20 &113 & 17   \\
 \midrule
\multirow{4}{2mm}{7} 
   & 8  &157&21&170&22&151&21&106 & 17  \\   
   &  16&157&21&172&23&167&26&132 & 18   \\ 
   &  32&150&21&169&23&165&26&148 & 18  \\
   &  64&143&20&163&22&162&23&151 & 19  \\
 \midrule
\multirow{4}{2mm}{10} 
   & 8  &207&24&215&25&203&24&139&18 \\   
   &  16&200&24&216&25&216&27&168&20  \\ 
   &  32&192&24&213&25&213&27&186&19   \\
   &  64&182&23&205&24&207&25&190&21   \\
\bottomrule
\end{tabular}}
\vspace{2ex}  
\caption{\it \textbf{Vectorial reaction-diffusion equation.} 
  Number of PCG iteration counts for the 
  divergence-conforming HDG scheme \eqref{hdg-vrd}.} 
\label{tab:2} 
\end{table}

\begin{table}[ht]
\centering 
\resizebox{0.9\columnwidth}{!}
{%
\begin{tabular}{cc| cc|cc|cc|cccc}
\toprule
$k$& $N$  
   &\multicolumn{2}{c|}{
  $\tau_1=1, \tau_2=1$}
  &
  \multicolumn{2}{c|}{
  $\tau_1=1, \tau_2=10^4$}
  &
  \multicolumn{2}{c|}{
  $\tau_1=10^4, \tau_2=1$}
  &
  \multicolumn{2}{c}{
  $\tau_1=10^4, \tau_2=10^4$}
\tabularnewline
\midrule
  & 
  &\aasp
  &\dasp &\aasp
  &\dasp &\aasp
  &\dasp &\aasp
  &\dasp \\
\midrule
\multirow{4}{2mm}{1} 
   & 8  &16  & 11&38  & 25&26  & 17&12  & 6 \\   
   &  16&16  & 11&35  & 24&27  & 20&13  & 7  \\ 
   &  32&16  & 10&34  & 23&26  & 21&14  & 8   \\
   &  64&15  & 10&33  & 23&25  & 19&16  & 9   \\
 \midrule
\multirow{4}{2mm}{4} 
   & 8  &25  & 18&39  & 26&25  & 19&15  & 14\\   
   &  16&25  & 17&40  & 27&30  & 21&19  & 15 \\ 
   &  32&24  & 16&39  & 27&31  & 22&23  & 16  \\
   &  64&23  & 15&38  & 27&31  & 22&25  & 16  \\
 \midrule
\multirow{4}{2mm}{7} 
   & 8  &29  & 22&43  & 29&27  & 23&21  & 20 \\   
   &  16&28  & 20&43  & 29&33  & 24&25  & 20  \\ 
   &  32&28  & 19&42  & 30&34  & 25&28  & 19 \\
   &  64&26  & 18&42  & 30&33  & 24&29  & 19 \\
 \midrule
\multirow{4}{2mm}{10} 
   & 8  &34  & 25&49  & 34&31  & 26&26  & 23\\   
   &  16&33  & 24&47  & 34&36  & 27&30  & 23 \\ 
   &  32&33  & 22&47  & 33&38  & 27&33  & 22  \\
   &  64&31  & 21&44  & 32&37  & 27&34  & 22  \\
\bottomrule
\end{tabular}}
\vspace{2ex}  
\caption{\it \textbf{Generalized biharmonic equation, simply supported boundary
  condition.} 
  Number of PCG iteration counts for the 
  CIP-HDG scheme \eqref{hdg-bh}.} 
\label{tab:3} 
\end{table}

\begin{table}[ht]
\centering 
\resizebox{0.9\columnwidth}{!}
{%
\begin{tabular}{cc| cc|cc|cc|cccc}
\toprule
$k$& $N$  
   &\multicolumn{2}{c|}{
  $\tau_1=1, \tau_2=1$}
  &
  \multicolumn{2}{c|}{
  $\tau_1=1, \tau_2=10^4$}
  &
  \multicolumn{2}{c|}{
  $\tau_1=10^4, \tau_2=1$}
  &
  \multicolumn{2}{c}{
  $\tau_1=10^4, \tau_2=10^4$}
\tabularnewline
\midrule
  & 
  &\aasp
  &\dasp &\aasp
  &\dasp &\aasp
  &\dasp &\aasp
 &\dasp  \\
\midrule
\multirow{4}{2mm}{1} 
   & 8  &28  & 17&36  & 24&27  & 18&14  & 6  \\   
   &  16&37  & 23&39  & 27&27  & 20&14  & 8   \\ 
   &  32&50  & 34&48  & 33&26  & 20&16  & 11   \\
   &  64&67  & 49&62  & 48&27  & 23&22  & 15   \\
 \midrule
\multirow{4}{2mm}{4} 
   & 8  &46  & 28&46  & 29&25  & 19&16  & 14\\   
   &  16&60  & 41&58  & 37&30  & 22&20  & 17 \\ 
   &  32&77  & 53&74  & 50&33  & 23&27  & 20  \\
   &  64&102  &74&99  & 71&40  & 27&36  & 23  \\
 \midrule
\multirow{4}{2mm}{7} 
   & 8  &53  & 34  &52  & 33  &28  & 24&21  & 21 \\   
   &  16&69  & 46  &66  & 42  &34  & 25&24  & 22  \\ 
   &  32&90  & 60  &86  & 58  &36  & 26&32  & 22 \\
   &  64&117  & 80 &112  & 79 &44  & 30&42  & 26 \\
 \midrule
\multirow{4}{2mm}{10} 
   & 8  &57  & 39&54  & 36&32  & 27&26  & 24\\   
   &  16&73  & 49&70  & 45&37  & 28&30  & 26 \\ 
   &  32&95  & 64&91  & 60&39  & 29&34  & 27  \\
   &  64&126 & 87&119 & 84&47  & 31&45  & 28  \\
\bottomrule
\end{tabular}}
\vspace{2ex}  
\caption{\it \textbf{Generalized biharmonic equation, clamped boundary
  condition.} 
  Number of PCG iteration counts for the 
  CIP-HDG scheme \eqref{hdg-bh}.} 
\label{tab:4} 
\end{table}

\section{Conclusion}
\label{sec:conclude}
We applied 
the ASP theory to construct robust preconditioners for 
the  HDG schemes for three class of elliptic operators
with a low order term. Extension of ASP theory for HDG scheme for other 
elliptic operator is the subject of ongoing research. 
Robust preconditioning of HDG schemes for saddle point systems is 
the subject of a forthcoming paper.
\bibliography{fsi}
\bibliographystyle{siam}

\end{document}